\documentclass[DIV=classic,a4paper,10pt]{myart}

\KOMAoptions{DIV=last}
\AtBeginDocument{\AtBeginShipoutNext{\AtBeginShipoutDiscard}}
\usepackage{ucs}
\usepackage[utf8x]{inputenc} 
\usepackage{caption}
\usepackage{subcaption}

\renewcommand{\mid}{\,|\,}

\DeclareMathOperator*{\esssup}{ess\,sup}

\def \R{\mathbb{R}}
\def \P{\mathbb{P}}
\def \Q{\mathbb{Q}}
\def \N{\mathbb{N}}
\def \E{\mathbb{E}}

\def \L{\mathbb{L}}

\usepackage{accents}

\begin{document}
	
	\title{BSDEs driven by $|z|^2/y$ and applications to PDEs and decision theory}
	\thanks{We Thank Guy Barles, Mohammed Hassani, Fran\c cois Murat Youssef Ouknine
		and Hatem Zaag for helpful discussions.}
	

	\author[]{Khaled Bahlali}
	\author[]{Ludovic Tangpi}

	\abstract{
		Existence and uniqueness is established for a large class of backward
		stochastic differential equations which contain singular terms of the form
		$\pm|z|^2/y$. 
		The results are applied to investigate singular partial differential equations
		(PDEs) and to decision theory problems that cannot be studied using classical
		regular BSDEs. 
		The application to PDEs concerns the existence of viscosity solutions to  PDEs
		containing a singular term of the form $\pm|\nabla v|^2/v$ with rather weak
		assumptions on the regularity of the coefficients. 
		Such PDEs with singularity in the value process appear in several applications
		in physics and economics. 
		Regarding the application to decision theory, on the one hand, we use singular
		BSDEs to solve portfolio optimization problems with logarithm and power utility
		and non-trivial terminal endowment.
		Moreover, we derive existence and uniqueness of the general version of
		the non-Markovian Kreps-Porteus stochastic differential utility defined by
		\citet{Duf-Eps92} and constructed, in the Markovian case using PDE arguments by
		\citet{Duf-Lions}.  
	}

	\date{\today}

	\vskip 1cm
	\keyWords{Domination condition; singularity at zero; viscosity solutions;
		probabilistic representation, PDE with Neumann lateral boundary; decision theory
		in finance; convex duality.} 
	
	\vskip 1cm
	\keyAMSClassification{60H10, 60H20, 35K58, 35K67, 91G10.} 
	
	\setcounter{page}{0} 
	\maketitle
	

	\section{Introduction}
	
	Let $(\Omega, {\cal F}, \P)$ be a probability space carrying a $d$-dimensional
	Brownian motion denoted $W$ and equipped with the $\P$-completion of the
	filtration $(\mathcal{F}_t) :=\sigma(W_s, t\le s\le T)$ generated by $W$, with
	$T\in (0,\infty)$.
	We equip $[0,T]\times \Omega$ with the progressive $\sigma$-field.
	The goal of this paper is to give conditions for existence and uniqueness of
	quadratic backward stochastic differential equations (BSDEs) of the form
	\begin{equation}
		\label{eq:bsde intro}
		Y_t = \xi + \int_t^TH(u,Y_u,Z_u)\,du - \int_t^TZ_u\,dW_u
	\end{equation} 
	when $H$ has a possible singularity at zero, and to investigate a few
	applications notably in singular PDE theory. 
	Recall that a solution of equation \eqref{eq:bsde intro}  is an adapted process
	$(Y,Z)$ which satisfies \eqref{eq:bsde intro} such that $Y$ is continuous and
	$\int_0^T|Z_s|^2ds < \infty$ $\P$--a.s.
	The following proposition gives a flavor of the type of result we present in the
	present paper.
	Notice that this is the simplest (or canonical) form of the equation under
	consideration.
\begin{proposition}
	\label{aperitif}
	Let $\xi$ be an $\mathcal{F}_T$-measurable random variable such that $\xi>0$ $\P$-a.s. or $\xi<0$ $\P$--a.s.
	\begin{itemize}
		\item[(a)] If $|\xi|^3$ is integrable, then the BSDE
		\begin{equation}
		\label{eq:z2/y.intro}
			Y_t = \xi + \int_t^T \frac{|Z_u|^2}{Y_u}\,du -\int_t^TZ_u\,dW_u
		\end{equation}
		has a unique solution such that $\E[\sup_{0\le t\le T}|Y_t|^2] + \E[\int_0^T|Z_t|^2\,dt]<\infty$. 
		Moreover, uniqueness holds among solutions such that $Y^3$ is in the class $(D)$. 			
		\item[(b)] If $1/\xi$ is integrable, then the BSDE
		\begin{equation}
		\label{eq:z2/y.intro}
			Y_t = \xi - \int_t^T \frac{|Z_u|^2}{Y_u}\,du - \int_t^TZ_u\,dW_u
		\end{equation}
		has a solution and uniqueness holds among solutions such that $1/Y$ is in the class $(D)$.
		If moreover $\xi$ is square integrable, then there exists a unique solution such that  $\E[\sup_{0\le t\le T}|Y_t|^2] + \E[\int_0^T|Z_t|^2\,dt]<\infty$.
	\end{itemize}
\end{proposition} 
We emphasize that the terminal condition $\xi$ is not assumed bounded, and can be arbitrarily close to zero.
The interest in backward SDEs with generators of the form $H(y,z) = |z|^2/y$ first arose from the work of \citet{Duf-Eps92} who analyzed existence and uniqueness of stochastic differential utilities.
In fact, due to the singularity at zero, the probabilistic method developed in that paper did not cover the utility model suggested by \citet{Kreps-Port78}, see subsection \ref{sec:SDU} for details.
It should be noted that the Kreps-Porteus utility in its differential form is an important class of utility functions, allowing for instance (when it exists) to develop a two-factor capital asset pricing model, see \cite{Duf-Eps92}.
In the Markovian case, the existence of the Kreps-Porteus stochastic differential utility has been proved by \citet{Duf-Lions} by deriving (Sobolev) solutions of a parabolic partial differential equation (PDE) of the form
\begin{equation}
\label{eq:pde intro}
	\begin{cases}
		\partial_tv + \frac 12\Delta v + \delta \frac{|\nabla v|^2}{v}= 0, \\
		v(T,x) = h(x)
	\end{cases}
\end{equation}
with the gradient and Laplacian operators acting on the spacial variable.  
As already pointed out in \cite[Section 2.4]{Duf-Lions}, the case $\delta<-1$ is particularly delicate.
This is materialized in the present probabilistic setting by the fact that the change of variable that we use "remove" the singularity is not well-defined when $\delta<-1$, see the proof of Proposition \ref{prop:-absury}. 

The main objective of the present paper is to develop a probabilistic method allowing to prove existence of quadratic singular BSDEs.
As a byproduct, this allows us to solve equations of the form \eqref{eq:pde intro} in the viscosity sense under  rather weak regularity conditions on the data.
Furthermore, a probabilistic representation of such PDEs is also given.
As it turns out, the relevance of equations \eqref{eq:z2/y.intro} and \eqref{eq:pde intro} is not restricted to the stochastic differential utility theory and economics.
They also naturally appear in quantitative finance and in physics.
	
In quantitative finance, BSDEs with generators of the form $|z|^2/y$ appear in problems of optimal investment and decision theory, see e.g. \citet{Nutz2012}, \citet{optimierung}, \citet{Xing17}, \citet{kharroubial} and \citet{epstein02}, but also in interest rates problems, see \citet{Hyndman09}.
We also refer to Subsection \ref{sec:utility max} below where we discuss a portfolio optimization problem.
The gist being that, while quadratic equations (generators of the form $|z|^2$) allow to solve exponential utility problems the case of power or logarithmic utility with non-zero terminal endowment require to solve singular equations with generators of the form $|z|^2/y$.
Thus, our results allow to study such utility maximization	problems which cannot be studied by the simpler BSDEs driven by $|z|^2$.

In physics, equations of the form \eqref{eq:pde intro} appear in modeling of quenching problems, see e.g. \citet{Dong-Levi89}, \citet{Merle92}, \citet{Mer-Zaag97} and \citet{Chap-Hun-Ock} for details.
These equations are also used to model gas flow in porous media, see e.g. \citet{Giachetti2012},   \citet{Giac-Mur09}.
We emphasis that the existing literature deals with existence (and properties) of \emph{Sobolev solutions} in the $H^1$-sense.
Note in passing that the link between BSDEs and Sobolev solutions of parabolic PDEs was considered by \citet{Bar-Les97} for BSDEs with Lipschitz generators and by Bahlali et al \cite{beh2015} in the case when the generator in of superlinear growth.

As mentioned above, the generators in the statement of Proposition \ref{aperitif} and the nonlinearity in \eqref{eq:pde intro} are the "canonical forms" of generators we consider.
Actually, it will be assumed that the generator $H$ is a continuous function satisfying a bound of the form 
\begin{equation}
\label{eq:H intro}
	0\le H(t,\omega,y,z) \le \alpha_t(\omega) +\beta_t(\omega)y + \gamma_t(\omega)z \pm \frac{\delta}{y}|z|^2 \quad \text{for } y>0.
\end{equation}
In particular, the generators will not necessarily be locally Lipschitz	continuous (in some cases can even be discontinuous in $y$). 
As a consequence, the problem will not be amenable to techniques involving Picard iterations, monotone stability, a priori estimates or localizations and
approximations which are prevalent in the literature.
The method we develop in the present article is rather based on a combination of a simple change of variable technique akin to Zvonkin's transform in the theory of stochastic differential equations and a domination argument developed in 
\cite{Bahlali_domi}. 
	
Also due to the singularity of the generator, the proof we give of uniqueness for generators of the form \eqref{eq:H intro} does not follow the well-trodden paths of comparison principles or Banach fixed point theorem as customary.
Note however that some form of comparison still holds in the "canonical cases" of Proposition \ref{aperitif}.
We instead base our arguments on convex duality techniques for BSDEs initiated by \citet{tarpodual}.
This requires an additional convexity condition on the generator.

Let us now say a few words on the extensive literature on BSDEs and their connections to parabolic PDEs.
When the generator $H$ is Lipschitz continuous (in $(y,z)$) and $\xi$ is square integrable, \citet{peng01} proved existence and uniqueness of a square integrable solution $(Y,Z)$.
The case where the generator can have quadratic growth in $z$ (i.e. grows slower than $|z|^2$) is particularly relevant in several applications.
It has initially been investigated by \citet{Duf-Eps92} and then by \citet{kobylanski01} for bounded terminal conditions $\xi$ by developing a monotone stability method.
As commonly assumed in the literature, all the above mentioned works assume $H$	continuous in $(y,z)$, or even locally Lipschitz, and that the terminal condition is bounded, or has exponential moments.
To the best of our knowledge, the only papers dealing with existence of quadratic BSDEs with generators dominated by $f(y)|z|^2$ (with $f$ not constant) are the papers of \citet{beo2017}, \citet{epstein03} and \citet{Bahlali_domi}.
In the works of \citet{beo2017} and \citet{epstein03}, the function $f$ is	assumed to be globally integrable (and even continuous in the second paper) and in \citet{Bahlali_domi}, it is assumed to be locally bounded. The case $f$ globally integrable is also considered in \citet{Bahlali_domi} with merely $\L^1$-integrable terminal value.

The case  $f(y) = \pm 1/y$ considered in the present article cannot be covered	by the above cited works since this function is not integrable, even locally.
This case turns out to be of special practical interest, and its treatment is more delicate.  
		
In the next section we state the main existence and uniqueness results for BSDEs of the form $|z|^2/y$.
The proofs are given in Section \ref{sec:proofs}.
The subsequent section \ref{sec:PDE} is devoted to applications to viscosity solutions of	singular parabolic PDEs.
We start by the case where there are no boundary conditions and then we consider singular PDEs with lateral Neumann boundary conditions.
In Section \ref{sec:decision}, we discuss applications to finance and economics.
In the appendix we prove seemingly new stability results for (forward)	stochastic differential equations with non-Lipschitz coefficients that are of independent interest.
These are necessary for the existence of viscosity solutions of PDEs with non-Lipschitz coefficients treated in Section \ref{sec:PDE}.
	
\section{Main results}
\label{sec:results}
	
Consider the following spaces and norms:
For $p>0$, we denote by \ $\L_{loc}^{p}(\R)$ the space of (classes) of functions $u$ defined on $\mathbb{R}$ which are $p$-integrable on bounded subsets of $\R$.
We also denote, $\mathcal{W} _{p,\,loc}^{2}(\R)$ the Sobolev space of (classes) of functions $u$ defined on $\mathbb{R}$ such that both $u$ and its generalized derivatives $u^{\prime }$ and $u^{\prime \prime }$ belong to ${\L}_{loc}^{p}(\mathbb{R})$.
By, $\mathcal{C}$ we denote the space of continuous and $\mathcal{F}_{t}$--adapted processes.
By $\mathcal{S}^{p}(\mathbb{R})$ we denote the space of continuous, $\mathcal{F}_{t}$--adapted processes $Y$ such that \ $\mathbb{E}\sup_{0\leq t\leq T}|Y_{t}|^{p} < \infty$, and ${\cal S}^\infty(\mathbb{R})$ the space of processes $Y \in {\cal S}^p(\mathbb{R})$ such that $\sup_{0\le t\le T}|Y_t|\in \L^\infty$. The set ${\cal S}^p_+(\mathbb{R})$ denotes the positive elements of ${\cal S}^p(\mathbb{R})$.
Let	$\mathcal{M}^{p}(\mathbb{R}^d)$ be the space of $\R^d$-valued $\mathcal{F}_{t}$--progressive processes $Z$ satisfying $\mathbb{E}\Big[\Big(\int_{0}^{T}|Z_{t}|^{2}dt\Big)^{\frac{p}{2}}\Big]<+\infty$.
By $\mathcal{L}^{2}(\mathbb{R}^d)$ we denote the space of $\mathcal{F}_{t}$--progressive processes $Z$ satisfying $\int_{0}^{T}|Z_{s}|^{2}ds<+\infty \ \mathbb{P}$\text{--a.s.} 
BMO is the space of uniformly integrable martingales $M$ satisfying
$$\sup_\tau||\E\big[|M_T-M_\tau|\mid{\cal F}_\tau\big]||_{\infty}<\infty$$
where the supremum is taken over all stopping times $\tau$ with values in $[0,T]$.
A process $Y$ is said to belong to the class (D) if the set $\{Y_\tau: \tau \text{ stopping time in } [0,T]\}$ is uniformly integrable. 
\begin{definition}
	\label{SolutionL2} Given $\xi \in \L^0$ and a progressively-measurable function $H:[0,T]\times \Omega\times \mathbb{R}\times \mathbb{R}^d\to \mathbb{R}$, we denote by BSDE$(\xi , H)$ the BSDE with terminal condition $\xi$ and generator $H$. A solution to BSDE$(\xi, H)$ is a process $(Y,Z)$ which belongs to $\mathcal{C}\times \mathcal{L}^{2}(\R^d)$ such that $(Y,Z)$ satisfies BSDE$(\xi , H)$ for each $t\in [ 0,T]$ and \ $\int_0^T |H(s, Y_s, Z_s)|ds < \infty$ \ $\P$--a.s.
\end{definition} 
Our first main result gives existence of BSDEs having generators with growth of	the form $|z|^2/y$.
Let $\delta\ge0$, and $\alpha,\beta:[0,T]\times \Omega\to \R_+$ and $\gamma :[0,T]\times \Omega\to \R^d$  be progressively measurable processes. 
Throughout the paper, the function $g$ is given by 
\begin{equation}
\label{eq:def g}
	g(t, \omega,y, z) := \alpha_t(\omega) + \beta_t(\omega)y + \gamma_t(\omega)z + \frac{\delta}{y}|z|^2.
\end{equation} 	

Consider the following conditions:
\begin{enumerate}[label = (\textsc{A1}), leftmargin = 30pt]
	\item $\alpha \in {\cal S}^2_+(\mathbb{R})$,  $\xi>0$ and there is $p>1$ such that
	$$\E\left[\xi^{(2\delta + 1)p}e^{p\int_0^T\lambda_u\,du }\right]<\infty\quad
		\text{and}\quad \E\Big[\Big(\int_0^Te^{\frac 12
			\int_0^s\lambda_u\,du}\alpha_s\,ds\Big)^p \Big]<\infty,$$
	with $\lambda_t:=(2\delta + 1)(\alpha_t+\beta_t) + \frac{|\gamma_t|^2}{2r}$ for	some $r\in (0,\frac{1\wedge(p-1)}{2})$,  $e^{\int_0^T\gamma_s\,dW_s}\in \L^q$, $\beta \in {\cal S}^q_+(\mathbb{R})$ and $|\gamma|\in {\cal	S}^{2q}(\mathbb{R})$, with $1/p + 1/q=1$. 
	\label{a2}
\end{enumerate}
\begin{enumerate}[label = (\textsc{A2}), leftmargin = 30pt]
	\item $\alpha \in {\cal S}^2_+(\mathbb{R})$, $\xi>0$ and there is $p>1$ such that
	$$e^{\int_0^T(\alpha_u + \beta_u)\,du}\xi\in \L^{(2\delta + 1)p}\quad
		\text{and}\quad  \int_0^Te^{\int_0^s(\alpha_u+\beta_u)\,du }\alpha_s\,ds \in
		\L^p$$
	$e^{\int_0^T\gamma_s\,dW_s}\in \L^q$, $\beta \in {\cal S}^q_+(\mathbb{R})$ and $|\gamma|\in {\cal S}^{2q}(\mathbb{R})$, with $1/p + 1/q=1$.\label{a3} 
\end{enumerate}
\begin{enumerate}[label = (\textsc{A1}'), leftmargin = 30pt]
	\item $\xi<0$ and satisfies the integrability conditions stated in \ref{a2}	along with $\alpha, \beta$ and $\gamma$. \label{a2prime}
\end{enumerate}
\begin{enumerate}[label = (\textsc{A2}'), leftmargin = 30pt]
	\item $\xi<0$ and satisfies the integrability conditions stated in \ref{a3}	along with $\alpha, \beta$ and $\gamma$. \label{a3prime}
\end{enumerate} 

It is clear that the integrability condition \ref{a2} is stronger than \ref{a3}.
Our main existence result is the following:
\begin{theorem}
\label{thm:exist_BSDE}
	If $\xi>0$, $H$ satisfies $0\le H\le g$ and condition \ref{a2} holds, then the BSDE$(\xi, H)$ admits a solution $(Y,Z) \in {\cal S}^{(2\delta +1)p}(\R)\times \mathcal{L}^2(\R^d)$ such that $0< Y\le Y^g$, where $(Y^g, Z^g)$ solves BSDE$(\xi, g)$.
		
	If condition \ref{a2} is replaced by \ref{a3}, then the solution $(Y,Z)$ satisfies $\sup_{t\in[0,T]}\E[|Y_t|^{2\delta +
			1}]<\infty$ and $Z \in {\cal L}^2(\R^d)$.
\end{theorem}
The proof is given in Subsection \ref{sec:exists} below.
We will also show that under slightly stronger conditions, BSDEs with generators driven by $|z|^2/y$ further admit unique solutions.
The precise statement is given in the following theorem:
\begin{theorem}
\label{thm:unique_BSDE}
	Assume that $\xi\in \L^\infty$, $\alpha, \beta \in {\cal S}^\infty(\mathbb{R})$, that $\int\gamma\,dW$ is a BMO martingale and that for each $(t,\omega)$, the function $H(t,\omega,\cdot,\cdot)$ is jointly convex.
	If $\xi>0$ and $H$ satisfies $0\le H\le g$, then for every solution $(Y, Z)$ of BSDE$(\xi, H)$ such that $0< Y\le Y^g$, the process $Y$ is bounded and moreover, if either $\delta<1/2$ or $\xi>c$ for some constant $c>0$, then $Z\in \mathcal{M}^2(\R^d)$.

	In addition, for every solutions $(Y,Z), (Y',Z')$ satisfying $0 < Y,Y' \le Y^g$,  the processes
		$Y$ and $Y'$ are indistinguishable and $Z = Z' dt\otimes \P$--a.s.
\end{theorem}
The sign conditions made on $H$ and $\xi$ are not necessary.
The following corollary can be treated using the method developed in the present article. 
	
\begin{corollary}
\label{cor:cor.main.exit}
	If $\xi<0$, $H$ satisfies $g\le H\le 0$ with $2\delta +1$ odd,  and the condition \ref{a2prime} is satisfied, then the BSDE$(\xi, H)$ admits a solution $(Y,Z)\in \mathcal{S}^{(2\delta +1)p}(\R)\times \mathcal{L}^{2}(\R^d)$ such that $Y^g\le Y<0$.
	If in addition $\xi\in \L^\infty$, $\alpha, \beta \in {\cal	S}^\infty(\mathbb{R})$, that $\int\gamma\,dW$ is a BMO martingale and that for each $(t,\omega)$, the function $H(t,\omega,\cdot,\cdot)$ is jointly concave	then BSDE$(\xi,H)$ admits a unique solution $(Y,Z)$ such that $Y^{g}\le Y<0$.
\end{corollary} 

The following third main result deals with BSDEs driven by $-\delta z^2/y$. This kind of BSDEs has direct applications to mathematical finance.
In particular, it arises when studying the existence of Kreps-Porteus utility in continuous time.
This problem was considered by Duffie-Lions \cite{Duf-Lions} using partial differential equations. 
Since the treatment of this kind of BSDEs is different from that of equations driven by $\delta z^2/y$, it is discussed separately. The following statements hold.  Let
\begin{equation}
\label{eq:def f}
	f(t, \omega,y, z) :=  \beta_t(\omega)y + \gamma_t(\omega)z - \frac{\delta}{y}|z|^2, 
\end{equation} 
and consider the integrability condition
\begin{enumerate}[label = (\textsc{A1}''), leftmargin = 30pt]
	\item  $\xi$, $\beta$ and $\gamma$ satisfy
	\begin{equation*}
	 	\exp\Big(\int_0^T|-2\delta + 1||\beta_s| -\frac12|\gamma_s|^2\,ds + \int_0^T\gamma_s\,dW_s  \Big)\xi^{-2\delta + 1} \in \L^2.
	\end{equation*}  \label{a1second}
\end{enumerate} 

\begin{theorem}\label{thm:existmoinsdelta} 
	\begin{itemize}
		\item[(a)] If $\xi>0$, $H$ satisfies $0\le H\le f$ and condition \ref{a1second} is satisfied, then the BSDE$(\xi,
		H)$ admits a solution $(Y,Z)\in \mathcal{S}^{(-2\delta +1)}(\mathbb{R})\times\mathcal{M}^{2}(\mathbb{R}^d)$ such that $0< Y\le Y^{f}$ where $(Y^f, Z^f)$ solves BSDE$(\xi, f)$.
		If in addition $\xi\in \L^\infty$, $\beta \in {\cal S}^\infty(\mathbb{R})$, that $\int\gamma\,dW$ is a BMO martingale and that for
		each $(t,\omega)$, the function $H(t,\omega,\cdot,\cdot)$ is jointly convex then
		BSDE$(\xi,H)$ admits a unique solution $(Y,Z)$ such that $0<Y\le Y^{f}$.
		
		\item[(b)] If $\xi<0$, $H$ satisfies $f\le H\le  0$ with $-2\delta+1$ odd, and condition \ref{a1second} is satisfied, then the BSDE$(\xi, H)$ admits a solution $(Y,Z)\in \mathcal{S}^{(-2\delta +1)}(\mathbb{R})\times \mathcal{M}^{2}(\mathbb{R}^d)$ such that $Y^{f}\le Y<0$.
		If in addition $\xi\in \L^\infty$, $\beta \in {\cal S}^\infty(\mathbb{R})$, that $\int\gamma\,dW$ is a BMO martingale and that for each $(t,\omega)$, the function $H(t,\omega,\cdot,\cdot)$ is jointly concave then BSDE$(\xi,H)$ admits a unique solution $(Y,Z)$ such that $Y^{f}<Y\le 0$. 
	\end{itemize}	
\end{theorem}
Before going any further, let us make the following important remarks. 
\begin{remark} 		
	(i) The monotonicity condition $(y-\tilde y)(H(s,y,z) - H(s,\tilde y, z)) \le	-a|y - \tilde y|^2$ for some $a \in \mathbb{R}$ and every $(y,y',z)\in \mathbb{R}^2\times \mathbb{R}^d $ is often used to derive uniqueness of BSDEs, see e.g. \cite{dp}.
	This condition is not compatible with our framework. 
	For instance, our canonical generator  $H(y,z) = -z^2/y$  is not monotone in the above sense.
	However, it is jointly concave in $(y,z)$.  
	Therefore, corollary \ref{cor:cor.main.exit} suggests that in the present setting, convexity-type conditions in $(y,z)$ are a more suitable assumptions to obtain uniqueness than monotonicity in $y$. 
	
	(ii) BSDEs with generators of the form $-z^2/y$ are particularly relevant in economics, since they correspond to the Kreps-Porteus stochastic differential utility in continuous time given in \cite{Duf-Eps92}.  

	(iii) As the reader will see in the proofs, assuming that $a:=-2\delta+1$ or $a:=2\delta +1$ are odd when $\xi<0$ are needed to ensure that quantities of the form $x^{a}$ are well-defined for $x<0$.
\end{remark}
		
Along with the existence result, the above uniqueness theorem is crucial for the existence of viscosity solutions of a class of singular parabolic PDEs.
For instance, our results will allow to solve the PDE, for $ \ \theta=\pm1$,  
\begin{equation*}
	\begin{cases}
		\frac{\partial {v}}{\partial s}(s,x)+b(t,x)\nabla_xv(s,x) +  \frac 12\Delta	v(s,x) +  \theta \frac{|\nabla_{x}v|^{2}}{v}(s,x) = 0, \quad  \text{on}\quad [0,T)\times \mathbb{R}^{d},\\
		v(t,x) > 0  \quad \text{on}\quad [0,T)\times \mathbb{R}^{d},\\
		v(T,x)= h(x) 
	\end{cases}
\end{equation*}
when the functions $b$ and $h$ are merely continuous in $x$, with additional growth conditions, see Theorem \ref{thm:1/y viscosity} and Remark \ref{rem:no regula}.
This will be developed in Subsection \ref{sec:parab PDE} for parabolic PDEs with no boundary conditions and Subsection \ref{sec:neumann} for the case of PDEs with lateral Neumann boundaries.
In Section \ref{sec:decision} we provide two applications of the existence	results to economics, namely to (non-exponential) utility maximization with random endowment and to the existence of Kreps-Porteus stochastic differential utility in continuous time.  
This part of our work can be seen as a probabilistic approach to the result of \citet{Duf-Lions}.
Moreover, we give existence of viscosity solutions of these singular PDEs along with a probabilistic representation.

\section{Proofs of the main results}
\label{sec:proofs}
	The idea of the proof of the existence results is to first focus on existence of
	equations with generators of types $g$ and $f$.
	To solve such equations, we will use change of variables allowing to reduce these
	equations to much simpler ones, in some cases to linear equations.
	Then, we will use the so-called domination property recalled in Lemma \ref{lem:dom} to obtain existence of the equation with generator $H$.
	We start by presenting preliminary results that we will bring together to prove	the main existence theorem.
	\subsection{Preliminaries}
	The following lemma summarizes properties of a simple linear BSDE which will
	play an important role in our approach. These properties are well-known, short
	proofs are given for completeness.	
	
	\begin{lemma}\label{rk0}
		Let $\xi$ be an $\mathcal{F}_T$-measurable random variable and
		$\gamma:[0,T]\times \Omega\to \mathbb{R}^d$ a progressively measurable process.
		Consider the BSDE
		\begin{equation}\label{0}
			Y_t = \xi +\int_t^T\gamma_sZ_s\,ds- \int_t^T Z_s\,dW_s.
		\end{equation}
		\begin{itemize}
			\item[(i)] If there is $p>1$ such that $\xi\in \L^p$ and
			$\exp\big(\int_0^T\gamma_s\,dW_s\big)\in \L^q$ with $1/p+ 1/q=1$, then \eqref{0}
			has a unique solution $(Y,Z)$ such that
			\begin{align*}
				\int_0^t Z_sdW_s \, \hbox{ is a local martingale and } \,
				\sup_{0\leq s \leq T} \E(|Y_s|) < \infty.
			\end{align*}
			\item[(ii)] If in addition $\xi$ is positive, then  the solution $Y_t$ is
			positive for each $t$.
			\item[(iii)] If $\xi =0$, then $(0,0)$ is a solution of \eqref{0}.  If $\xi
			\neq 0$, then for any process $Z$, the pair $(0, Z)$ cannot be a solution of
			\eqref{0}.
			\item[(iv)] If $\gamma=0$ and there is a positive solution $Y$ to equation \eqref{0}, then necessary $\xi\in \L^1$.
			\item[(v)] If $\gamma=0$ and $\xi\in \L^1$, then Equation \eqref{0} admits a
			unique solution $(Y,Z)$ such that $Y$ belongs to the class (D) and $Z \in {\cal
				M}^p(\mathbb{R}^d)$ for each $0<p<1$.
		\end{itemize}
	\end{lemma}
	Henceforth, denote by  $\Q^q$ the probability measure with density
	\begin{equation*}
		\frac{d\mathbb{Q}^q}{d\P}:=\exp\Big(\int_0^Tq_s\,dW_s -
		\frac{1}{2}\int_0^T|q_s|^2\,ds \Big)
	\end{equation*}
	where $q:[0,T]\times \Omega\to\mathbb{R}$ is a progressively measurable process
	such that $\int_0^T|q_u|^2\,du<\infty$.
	\begin{proof}
		$(i)$ Since the probability measure $\Q^\gamma$ is equivalent to $\P$, the existence
		in (i) follows by an extension of martingale representation theorem sometimes
		called Dudley's representation theorem, see e.g.  \cite[Theorem 12.1]{MSteel01}.
		Moreover, $\xi \in \L^1(\mathbb{Q}^\gamma)$ and it holds
		\begin{equation}
			\label{eq:Y linear eq}
			Y_t = \E_{\mathbb{\Q^\gamma}}[\xi \mid {\cal F}_t].
		\end{equation}
		This shows in particular, due to H\"older's inequality, that $\sup_{0\le t\le
			T}\E[|Y_t|]<\infty$.
		That $\int Z\,dW$ is a local martingale follows by Girsanov's theorem.
		
		$(ii)$ If $\xi>0$, then by \eqref{eq:Y linear eq} and the fact that
		$\mathbb{Q}^\gamma$ is equivalent to $\P$ we have $Y>0$.
		The latter argument further shows that if $\xi = 0$, then $Y =0$ and $Z=0$, and
		if $\xi \neq 0$ we must have $Y\neq 0$, which proves $(iii)$.
		
		$(iv)$ Now, assume $\gamma = 0$.
		If $(Y,Z)$ is a solution such that $Y>0$, then $\xi>0$.
		Let $\tau_n$ be a localizing sequence such that
		$\int_0^{\tau_n\wedge\cdot}Z\,dW$ is a martingale.
		Then, it holds $Y_0 = \E[Y_{\tau_n}]$ and by Fatou's lemma and continuity of
		$Y$, this implies $\E[\xi]\le Y_0<\infty$. 

		$(v)$Reciprocally, if $\xi \in \L^1$, the
		proof goes as  in \cite{Bahlali_domi}. 
		Finally, assertion $(v)$ follows by \cite[Proposition 1.1]{Bahlali_domi}.
	\end{proof}
	
	Another tool in our arguments is the so-called existence by domination result which we present below, see \cite{Bahlali_domi} for details.
	\begin{definition}\label{defdom} (Domination conditions) We say that  the data
		$(\xi, H)$  satisfy a domination condition if  there exist two progressively
		measurable processes $H_1$ and  $H_2$ and two $\mathcal{F}_T$-measurable random
		variables $\xi_1$ and $\xi_2$ satisfying
		\begin{itemize}
			\item[(B1)] $\xi_1 \leq  \xi \leq \xi_2$
			\item[(B2)] BSDE$(\xi_1, H_1)$ and BSDE$(\xi_2, H_2)$ have two solutions
			$(Y^1,Z^1)$ and $(Y^2,Z^2)$ respectively, such that:
			\begin{itemize}
				\item[(i)] $Y^1 \leq Y^2$,
				\item[(ii)] for every $(t,\omega)$,  $y\in [Y_t^1(\omega), \ Y_t^2(\omega)]$
				and $z\in \mathbb{R}^d$, it holds $H_1(t,y,z) \leq  H(t,y,z) \leq H_2(t,y,z)$
				and
				$
				|H(t, \omega, y, z )| \leq \eta_t(\omega)+ C_t(\omega)|z|^2.
				$
			\end{itemize}
		\end{itemize}
		where  $C$ and $\eta$ are $\mathcal{F}_t$-adapted processes such that $C$ is
		continuous and $\eta$ satisfies for each $\omega$, \ 
		$\int_0^T|\eta_s(\omega)|ds < \infty$.
	\end{definition}
	In what follows we denote by $\mathcal{C}(\mathbb{R})$ the space of adapted
	processes with (almost sure) continuous paths.
	\begin{lemma}\label{lem:dom}(Existence by domination)
		Let $H$ be continuous in $(y,z)$ for $a.e. \ (t,\omega)$. Assume moreover that
		$(\xi, H)$  satisfy the domination conditions (B1)--(B2). Then, BSDE$(\xi, H)$
		has at least one solution $(Y,Z)\in {\cal C}(\R)\times {\cal L}^2(\R^d)$ such
		that
		$Y^1 \leq Y \leq Y^2$. Moreover, among  all solutions which lie  between 
		$Y^1$ and  $ Y^2$, there exist a maximal and a minimal solution.
	\end{lemma}
	
	This lemma, whose proof can be found in \cite{Bahlali_domi}, is an intermediate
	value-type theorem.  It directly gives the existence of solutions. Neither a
	priori estimates nor approximations are needed.
	The idea of the proof consists in deriving the existence of solutions for the
	BSDE without reflection from solutions of a suitable quadratic BSDE with two
	reflecting barriers obtained by \cite[Theorem 3.2]{EH2013}, see also
	\cite{Ess-Has11}. Note that the result \cite[Theorem 3.2]{EH2013} is established
	without assuming any integrability conditions on the terminal value.

	\subsection{Existence}
	\label{sec:exists}
	
	This section is dedicated to the proof of Theorem \ref{thm:exist_BSDE}.
	We start by giving the argument for the proof of the first part of Proposition
	\ref{aperitif}, which follows as a particular case of the following result.
	The second part of Proposition \ref{aperitif} will be covered by the proof of
	Proposition \ref{prop:-absury}.

\subsubsection{The BSDE $\delta |z|^2/y$} 
	
\begin{proposition}
\label{pro:aperitif}
	Let $\delta \ge0$, and let $\xi$ be a random variable such that $\E[|\xi|^{2\delta  + 1}]<\infty$, with $\xi>0$ $\P$--a.s., put $H(t,y,z)=\delta|z|^2/y$.
	Then, the BSDE \eqref{SolutionL2} has a solution $(Y,Z)$ such that $\sup_{0\le t\le T}E[|Y_t|^{2\delta + 1}]<\infty$ and $Z \in {\cal L}^2$. 	
	The uniqueness holds among solutions $(Y,Z)$ such that $Y^{2\delta + 1}$ is in the class $(D)$.

	Moreover, for all $\delta\ge0$ the solutions have the integrability property $\E[\sup_{0\le t\le T}|Y_t|^2] + \E[\int_0^T|Z_t|^2\,dt]<\infty$. 
		
	If $2\delta + 1$ is an odd number, then the above results remain true when $\xi<0$ $\P$-a.s.
\end{proposition}
	
\begin{proof}
	The function
	\begin{equation}\label{u(Y-1}
		u(y) := \frac1{2\delta + 1} y^{2\delta + 1}
	\end{equation}
	is a  twice continuously differentiable function which is one to one from $\R$ onto $\R$. Moreover, its inverse $v:= u^{-1}$ is also twice continuously differentiable from $\R^*_+$ to $\R_+$ (and from $\mathbb{R}^*_-$ to $\mathbb{R}_-$ when $2\delta + 1$ is odd).  
	Therefore, Itô's formula shows that BSDE$(\xi , \frac\delta y|z|^2)$ has a solution  if and only if BSDE$(\frac1{2\delta + 1} \xi^{2\delta + 1}, 0)$ has a strictly positive (or a strictly negative) solution. 
	According to Dudley's representation theorem, see e.g. \cite[Theorem 12.1]{MSteel01} BSDE$(\frac1{2\delta + 1} \xi^{2\delta + 1}, 0)$ has a solution for any $\mathcal{F}_T$-measurable random variable $\xi $ (no integrability is needed for $\xi$). 
	But, in order to apply It\^o's formula to the function $u^{-1}(x) = ((2\delta + 1)x)^{\frac1{2\delta +1}}$, we need that BSDE$(\frac1{2\delta + 1} \xi^{2\delta +1}, 0)$ has a strictly positive (or strictly negative) solution. 
	This holds when $\xi$ belongs to $\L^{2\delta + 1}$ and,  $\xi > 0$ (or  $\xi < 0$ and $2\delta + 1$ is odd). In this case, according to Lemma \ref{rk0},
	BSDE$(\frac1{2\delta  + 1} \xi^{2\delta + 1}, 0)$ has a unique solution  $(\bar{Y}, \bar{Z})$ := $(\frac1{2\delta  + 1}Y^{2\delta + 1},\ Y^{2\delta}Z)$ such that $\bar{Y}$ belongs to class $(D)$ and $\bar{Z}$ belongs to $\mathcal{M}^p(\mathbb{R}^d)$ for each $0<p<1$.
	Putting $Y:= ((2\delta + 1)\bar Y)^{\frac 1{2\delta + 1}}$, there is $Z$ such	that $(Y,Z)$ solves BSDE$(\xi,  \frac \delta y |z|^2)$ and, for some constant $K\ge0$, we have
	\begin{equation}
	\label{eq:class D}
		\sup_{0\leq \tau \leq T}\E[|(Y_\tau)^{2\delta + 1}|] = \frac1{2\delta +1}\sup_{0\leq \tau \leq T}\E[|\bar{Y}_\tau|] \  \le K
	\end{equation}
	where the supremum is taken over all stopping times $\tau\leq T$.

	Let us now prove uniqueness.
	Assume $\xi>0$ and let $(\bar Y,\bar Z), (\tilde Y, \tilde Z)\in {\cal S}^2(\mathbb{R})\times {\cal M}^2(\mathbb{R}^d)$ be two strictly positive solutions of BSDE$(\xi, \frac{\delta}{y}|z|^2)$.
	Applying It\^o's formula to $u(Y_t)$ shows that there are progressively measurable processes $\tilde Z'$ and $\bar Z'$ such that $(\tilde Y', \tilde Z')$ and $(\bar Y', \bar Z')$ solve BSDE$(\frac 1{2\delta + 1}\xi^{2\delta + 1},0)$, with $\tilde Y_t' =  u(\tilde Y_t)$ and $\bar Y_t' = u(\bar Y_t)$. Since $\tilde Y'$ and $\bar Y'$ are of class (D), it follows by Lemma \ref{rk0} that $\tilde Y' = \bar Y'$ and $\tilde Z' = \bar Z'$.
	The case $\xi <0$ is proved analogously.
		
	To prove the integrability property we distinguish two cases.

	\emph{Case 1:} We first assume that $\delta \neq 1/2$.
	We shall show that $(Y,Z)$ belongs to $\mathcal{S}^{2}(\mathbb{R})\times\mathcal{M}^{2}(\mathbb{R}^d)$. Itô's formula gives
	\begin{equation*}
		|Y_t|^2 = |\xi|^2 + \int_t^T(2\delta - 1)|Z_s|^2\,ds -  2\int_t^T Y_s Z_s\,dW_s.
	\end{equation*}
	For $n>0$, let $\tau_n := \inf\{t\geq 0; \int_t^T|Y_s Z_s|^2ds \geq n\}\wedge T$.
	It holds $\E[|Y_0|^2] = \E[\xi^2] + (2\delta-1)\E \int_0^{T\wedge\tau_n}|Z_s|^2\,ds$.
	Since $\tau_n \to T$ as $n\to \infty$, we deduce that it holds
	\begin{equation}\label{bornezcarrey-1}
		\E \int_0^T|Z_s|^2ds \leq\frac{1}{|2\delta -1|} \E[\xi^2] + \E[|Y_0|^2].
	\end{equation}
	We now prove that $Y$ belongs to $\mathcal{S}^{2}(\mathbb{R})$. Using Itô's formula and Doob's inequality, it follows that there exists a universal constant $\ell$ such that for any $\varepsilon > 0$
	\begin{align*}
		\E\Big[\sup_{0\leq t \leq T}|Y_t|^2 \Big] & \leq \E[\xi^2] + (2\delta + 1)\E\int_0^T|Z_s|^2ds + 2\E\Big[\sup_{0\leq t \leq T}|\int_t^T Y_s Z_sdW_s|\Big]\\
			& \leq \E[\xi^2] + (2\delta + 1)\E\int_0^T|Z_s|^2\,ds + \frac{\ell}{\varepsilon}\E\Big[\sup_{0\leq t \leq T}|Y_t|^2\Big] + \varepsilon
			\E\int_0^T|Z_s|^2ds.
	\end{align*}
	Taking $\varepsilon = 4\ell$, we deduce that
	\begin{align*}
		\E\Big[\sup_{0\leq t \leq T}|Y_t|^2 \Big] \leq 2\E[\xi^2] + 2(4\ell + 2\delta+1)\E\int_0^T|Z_s|^2ds
	\end{align*}
	which shows that $Y$ belongs to $\mathcal{S}^{2}(\mathbb{R})$.

	\emph{Case 2:} We now assume $\delta = 1/2$.
	We have showed that there exists a solution $(Y,Z)$ such that $\sup_{0\le t\le T}E[|Y_t|^{2}]<\infty$,  $Z \in {\cal L}^2$ and $Y>0$. 
	By Itô's formula, we have 
	\begin{equation*}
		Y_t^2 = \xi^2 - 2\int_t^T Y_s Z_sdW_s.
	\end{equation*}
	That is, the process $(\bar Y, \bar Z) := (Y^2, 2YZ)$ is a solution to the BSDE  
	\begin{equation}\label{Ybar1}
		\bar Y_t = \bar\xi -  2\int_t^T \bar Z_sdW_s, \qquad t\in [0, \ T].
	\end{equation}
 	Since $\E(|\bar\xi|) < \infty$, then according to  \cite{beh2015} we have  \ $\E[\sup_{0\leq s \leq T} |\bar Y|] < \infty$ and hence, \ $\E[\sup_{0\leq s \leq T} | Y|^{2}] < \infty$.  
 	Let us now prove that $Z \in \mathcal{M}^2(\R^d)$.	
 	Since $Y > 0$, then Itô's formula applied to $Y^2 \ln( Y)$ gives 
	\begin{align*}\label{itoYlogY}
		Y_t^2\ln( Y_t) & = \xi^2 \ln( \xi) 
		- \int_t^T|Z_s|^2\,ds  - \int_t^T (2Y_s\ln(Y_s) + Y_s)Z_s \ dW_s
	\end{align*} 
	For an integer $n$, we put 
	\begin{equation}\label{taun}
		\tau_n = \inf\Big\{ s > 0: \, \ \int_0^s |Y_rZ_r|^2+ |Y_r \ln (Y_r) Z_r|^2dr \ \geq \ n\Big\}.
	\end{equation}  
	It is easily checked that $\tau_n$ tends to $\infty$ as $n$ tends to $\infty$.
	Since $Y$ is continuous and $Z$ belongs to $\mathcal{L}^2$, then by a suitable localization one can assume that the stochastic integral in formula \eqref{itoYlogY} is a uniformly integrable martingale. Therefore, 
	\begin{align}\notag
		\E\Big[\int_0^T|Z_s|^2 \, ds\Big] & \leq \E\Big[\sup_{0\leq T}|Y_t^2\ln( Y_t)|\Big]  +  \E[|\xi^2 \ln( \xi)|] \\  \label{Zcarre}
 		& \leq K\Big(1  +  \E[\sup_{0\leq T}|Y_t|^{2p}] +  \E[|\xi|^{2p}]\Big) \ < \ \infty 
	\end{align}
	for some positive constant  $K$
	Proposition \ref{aperitif} is proved.
\end{proof}  
	\subsubsection{The BSDE$(\xi,g )$ with $g(t,\omega, y,z):= \alpha_t(\omega) +
		\beta_t(\omega)y + \gamma_t(\omega)z +  \frac{\delta}{y}|z|^2$ }	
	
In this subsection the quadratic BSDE under consideration is
\begin{equation}\label{abunsury}
	Y_{t}=\xi +\int_{t}^{T}g(s, Y_s, Z_s)ds-\int_{t}^{T}Z_{s}dW_{s},\, \ \ 0\leq t\leq T
\end{equation} 
where $g$ is defined by \eqref{eq:def g}
	
\begin{proposition}\label{propabsury}
	Let $\delta\ge0$ be fixed. 
	The following statements hold:
	\begin{itemize}
		\item[(a)] If $\xi>0$, $g\ge 0$ on $[0,\ T]\times\mathbb{R}_+\times\mathbb{R}^d$ and
		condition \ref{a2} holds, then the BSDE \eqref{abunsury} has a solution  $(Y,Z)$ such that $Y \in \mathcal{S}^{(2\delta +1)p}(\R)$, $Y>0$ and $Z\in {\cal L}^2(\R^d)$. 
		
		If moreover, $\xi \in \L^{2p+\lambda}$ for some $\lambda>0$, then for every $\delta\geq 0$ BSDE$(\xi, g)$ has a solution $(Y,Z)\in {\cal S}^{2p+\lambda}(\R)\times{\cal M}^2(\mathbb{R}^d)$.
			
		\item[(b)] If $\xi<0$, $g\le 0$ on $[0,\ T]\times\mathbb{R}_-\times\mathbb{R}^d$ with $2\delta + 1$ being an odd number and condition \ref{a2prime} holds, then the BSDE \eqref{abunsury} has a solution $(Y,Z)$ such that $Y \in \mathcal{S}^{(2\delta +1)p}(\R)$, $Y<0$ and $Z\in {\cal L}^2$.
		If $\delta\neq 1/2$, then we have $Z\in {\cal M}^2(\mathbb{R}^d)$. 
	\end{itemize}
\end{proposition} 
	
\begin{proof} 
	We consider only the case  $\xi > 0$, i.e. (a). 
	The second case goes  similarly.
	The BSDE \eqref{abunsury} admits a solution if and only if the BSDE
	\begin{equation}
	\label{eq:bsde_firstchange}
		\bar Y_t = \bar\xi + \int_t^T \left(\alpha_s\big\{(2\delta +1)\bar Y_s\big\}^{\frac{2\delta}{2\delta+1}} + (2\delta + 1)\beta_s\bar Y_s + \gamma_s\bar Z_s\right) \,ds - \int_t^T\bar Z_s\,dW_s
	\end{equation}
	also does, where $\bar \xi:=\frac 1{2\delta + 1} \xi^{2\delta + 1}$.
	This follows from It\^o's formula applied to $\frac{1}{2\delta + 1}Y^{2\delta + 1}$ and $\{(2\delta + 1)\bar Y\}^{1/(2\delta + 1)}$.
	Note that in order to apply It\^o's formula to $\{(2\delta + 1)\bar	Y\}^{1/(2\delta + 1)}$, we need that $\bar Y$ be strictly positive (or strictly negative, in which case $2\delta + 1$ should be an odd integer).
	We will show below that this holds when $\xi>0$ (or $\xi<0$).
		
	In order to prove existence of a solution to \eqref{eq:bsde_firstchange}, we	will apply Lemma \ref{lem:dom} with $\xi_1 = \xi_2 = \bar{\xi}$, $H_1(t,z)=\gamma_tz$, $ H(t, y,z) = \alpha_s\big\{ (2\delta +1)y \big\}^{\frac{2\delta}{2\delta+1}} + (2\delta + 1)\beta_ty + \gamma_tz$ and $H_2(t, y, z) = (2\delta + 1)\big( \alpha_t + (\alpha_t + \beta_t)y \big) + \gamma_t z$.
	Since $\xi^{2\delta + 1} \in \L^p$, and $\exp\big(\int_0^T\gamma_t\,dW_t\big)\in \L^q$, we clearly have
	\begin{itemize}
		\item	$Y^{\prime}_t = \mathbb{E}_{\Q^\gamma}\Big[\bar{\xi}\mid \mathcal{F}_t\Big]$ is a solution of BSDE$(\bar\xi,H_1)$ and $Y'>0$,	
		\item $H_1(t,z) \le H(t,y,z)\le H_2(t,y,z)$ for every  $y \geq 0$ and every $z\in \mathbb{R}^d$.
	\end{itemize}
	By \cite[Theorem 2.1]{beh2015} the BSDE$(\xi_2,H_2)$ admits a unique solution
	$(Y^{\prime\prime}, Z^{\prime\prime})\in {\cal S}^p(\mathbb{R})\times{\cal M}^p(\mathbb{R}^d)$.
	Since $\alpha$ and $\beta$ are positive, we have $Y^{\prime\prime}\ge Y^{\prime}$ and since $\bar\xi>0$ (which is equivalent to $\xi>0$), we further have $Y^{\prime}>0$.
	It follows from Lemma \ref{lem:dom} that Equation \eqref{eq:bsde_firstchange} admits a solution $(\bar Y, \bar Z)$ such that $Y^{\prime}\le \bar Y\le Y^{\prime\prime}$.
	In particular, $\bar Y \in {\cal S}^p(\mathbb{R}) $.
	Consequently, Equation \eqref{abunsury} admits a solution $(Y,Z)$ with $Z \in	{\cal L}^2(\mathbb{R}^d)$, where $\E\sup_{0\leq t \leq T}|Y_t|^{(2\delta +1)p} = (2\delta  +1)\E\sup_{0\leq t \leq T}|\bar{Y}_t|^p < \infty$ and $Y>0$.
		
	Let us now show the integrability property.
	As in the proof of Proposition \ref{abunsury}, we distinguish two cases:

	\emph{Case 1:}
	Assume $\delta \neq 1/2$. Let us show that in this case $Z \in {\cal M}^2(\mathbb{R}^d)$.
	Since $(Y,Z)$ satisfies Equation \eqref{abunsury} and $Y>0$, applying It\^o's formula to $Y^2$   yields
	\begin{equation*}
		Y^2_t = \xi^2 + \int_t^T2\alpha_sY_s + 2 \beta_sY^2_s + 2Y_s\gamma_sZ_s + (2\delta -1)Z^2_s\,ds - 2\int_t^TY_sZ_s\,dW_s.
	\end{equation*}
	Let $n \in \mathbb{N}^*$ and consider the stopping time
	\begin{equation*}
		\tau_n:= \inf\Big\{t>0: \int_0^t2|Y_s|^2|Z_s|^2\,ds>n \Big\}\wedge T.
	\end{equation*}
	Then, there is a constant $C\ge0$ such that for every $n \in \mathbb{N}$ and every $\varepsilon>0$ we have
	\begin{align*}
		\E\Big[\int_0^{\tau_n}|Z_s|^2\,ds \Big] &\le \frac{1}{|2\delta - 1|}Y^2_0+	\frac{1}{|2\delta - 1|}\E\Big[Y^2_{\tau_n} + \int_0^{\tau_n}2\alpha_uY_u + 2\beta_uY^2_u + 2Y_u\gamma_uZ_u\,du \Big]\\
		& \le CY^2_0+ C\E\Big[Y^2_{\tau_n} + \int_0^{\tau_n}\alpha_u^2 + Y_u^2 +	\frac 2q \beta_u^q + \frac 2p Y^{2p}_u + \frac {1}{p\varepsilon}Y_u^2p +
		\frac{\varepsilon}{2}|Z_u|^2+ \frac{1}{q\varepsilon}|\gamma_u|^{2p}\,du \Big]
	\end{align*}
	where the inequality follows by H\"older's inequality.
	This shows that there is a constant $C>0$ such that
	\begin{equation*}
		\E\Big[\int_0^{\tau_n}|Z_s|^2\,ds \Big] \le C\Big( 1 + \sup_{t\in[0,T]}\E[Y^{2p}_t]+ \E\Big[Y^2_{\tau_n}+ \int_0^T\alpha^2_u + \beta^q_u + |\gamma_u|^{2q}\,du \Big] \Big).
	\end{equation*}
	Taking the limit in $n$ on both sides, it follows by continuity of $Y$ and Lebesgue dominated convergence theorems that 
	\begin{equation*}
		\E\Big[\int_0^{T}|Z_s|^2\,ds \Big] \le C\Big( 1 + \sup_{t\in [0,T]}\E[Y^{2p}_t]+ \E\Big[\xi^2+ \int_0^T\alpha^2_u + \beta^q_u + |\gamma_u|^{2q}\,du \Big] \Big)<\infty.
	\end{equation*}
	That is, $Z\in {\cal M}^2(\mathbb{R}^d)$.

	\emph{Case 2:} Assume $\delta=1/2$.
	According to the first part of the proof, Proposition \ref{propabsury} (a), BSDE$(\xi,g)$ has a solution $(Y,Z)$ such that $Y>0$ and  satisfies $\E[\sup_{0\leq t \leq T}|Y_t|^{2p+\lambda}]<\infty$ for some $\lambda>0$.
	We shall prove that $Z \in {\cal M}^2(\mathbb{R}^d)$.  
	Since $Y > 0$, then Itô's formula applied to $Y^2\ln( Y)$ gives 
	\begin{align*}\label{itoYlogY}
		Y_t^2\ln(Y_t) = \xi^2 \ln(\xi) + \int_t^T (\alpha_s + \beta_sY_s  + \gamma_sZ_s)(2Y_s\ln(Y_s) + Y_s) ds- \int_t^T|Z_s|^2ds  - \int_t^T (2Y_s\ln(Y_s) +Y_s)Z_s \ dW_s.
	\end{align*} 
	Since $Y$ is continuous and $Z$ belongs to $\mathcal{L}^2(\R^d)$, then by a suitable localization one can and assume that the stochastic integral in the previous formula is a uniformly integrable martingale. Therefore, 
	\begin{align*}\label{itoYlogY}
	\E\Big[\int_0^T|Z_s|^2\,ds\Big] 
		& \leq \E|Y_0^2\ln( Y_0)|) + \E[|\xi^2 \ln( \xi)|]\\ 
		& \quad \ + \E\Big[\int_0^T 2\alpha_sY_slnY_s +\alpha_sY_s + \beta_sY_s^2\ln(Y_s) + \beta_sY_s^2  +2\gamma_sZ_sY_s\ln(Y_s)  + \gamma_sZ_sY_s \,ds\Big].
	\end{align*} 
	Arguing as in the proof of Proposition \ref{propabsury} and using the fact that  $|y\ln(y)| \leq K(1+ |y|^{1+\varepsilon})$ for any $\varepsilon >0$ and any $y$, we show that for any $\varepsilon > 0$,
	\begin{equation*}
		\E\Big[\int_0^T|Z_s|^2 \,ds\Big] \leq 2\E \Big[\sup_{0\leq T}|Y_t^2\ln( Y_t)|\Big] + K\E\bigg[\int_0^T \Big(1 + \alpha_s^2 +  \beta_s^{q} + \frac{1}{\varepsilon}|\gamma_s|^{2q} + |Y_s|^{2p(1+\varepsilon)} \Big)\,ds\bigg]   
	\end{equation*} 
	where  $K$ is some positive constant which depends on $p$ and $q$.  
	To complete the proof, it suffices to take  $\varepsilon = \frac{\lambda}{2p}$. 
	This concludes the proof.
\end{proof}


\begin{remark}\label{rkpropabsury}  
	$(i)$ \ If in assertion $(a)$ of the previous proposition, \ref{a2} is	replaced by the weaker condition \ref{a3} and $\delta\neq1/2$, then the solution $(Y,Z)$ satisfies the integrability   $\sup_{t\in[0,T]}E[|Y_t|^{2\delta + 1}]<\infty$ and $Z \in {\cal M}^2(\mathbb{R}^d)$. 
		
	$(ii)$ \ If in assertion $(b)$ of the previous proposition, \ref{a2prime} is replaced by \ref{a3prime} and $\delta\neq1/2$, then the solution $(Y,Z)$
		satisfies the integrability \ $\sup_{t\in[0,T]}E[|Y_t|^{2\delta + 1}]<\infty$ and $Z \in {\cal M}^2(\mathbb{R}^d)$.
\end{remark}
\begin{proof}
	In fact, in case the integrability condition \ref{a2} is replaced by the condition \ref{a3}, we notice that the BSDE$(\xi_2, H_2)$ introduced in the proof of Proposition \ref{propabsury} admits a solution $(Y'', Z'')$ if and only if the equation
	\begin{equation}
	\label{eq:BSDE_secondchange}
		\widetilde Y_t = e^{(2\delta + 1)\int_0^T(\alpha_u + \beta_u)\,du}\bar\xi +	\int_t^T(2\delta + 1)e^{(2\delta + 1)\int_0^s(\alpha_u + \beta_u)\,du}\alpha_s + \gamma_s\widetilde Z_s\,ds - \int_t^T\widetilde Z_s\,dW_s
	\end{equation}
	also does.
	This follows again as application of It\^o's formula to the processes $Y''_t =e^{-(2\delta + 1)\int_0^t(\alpha_u + \beta_u)\,du}\widetilde Y_t$ and $\widetilde Y_t = e^{(2\delta + 1)\int_0^t(\alpha_u + \beta_u)\,du} Y''_t$.
	Moreover, since there is $p>1$ such that it holds $e^{(2\delta + 1)\int_0^T(\alpha_u + \beta_u)\,du}\xi^{2\delta + 1} \in \L^p$ and $\exp\big(\frac 12 \int_0^T\gamma_t\,dW_t\big)\in \L^q$, it follows that Equation \eqref{eq:BSDE_secondchange} admits a solution $(\widetilde Y, \widetilde Z)$ which is given by
	\begin{equation*}
		\widetilde Y_t = \E_{\mathbb{Q}^\gamma}\Big[e^{(2\delta + 1)\int_0^T(\alpha_u + \beta_u)\,du}\bar\xi + \int_t^T(2\delta + 1)e^{(2\delta + 1)\int_0^s(\alpha_u+ \beta_u)\,du}\alpha_s\,ds \mid {\cal F}_t\Big].
	\end{equation*}
	Now, using H\"older's inequality we get $\sup_{0\le t\le T}\E[\widetilde Y_t]<\infty$, and by Girsanov's theorem it holds $\widetilde Z \in {\cal L}^2(\mathbb{R}^d)$.
	Thus, since $\alpha$ and $\beta$ are positive, BSDE$(\xi_2, H_2)$ admits a	solution $(Y'', Z'')$ satisfying $Y''_t \ge Y'_t$.
	By $\bar{\xi}>0$ (which is equivalent to $\xi > 0$), we further have that $Y''>0$.
	Hence, $0\le Y''\le \widetilde Y$, showing that $\sup_{0\leq t \leq T}\E[|Y''_t|^{(2\delta + 1)}] < \infty$ and arguing as above we get $\sup_{0\leq t \leq T}\E[|Y_t|^{(2\delta + 1)}] < \infty$ and $Z \in {\cal M}^2(\mathbb{R}^d)$. The case $\xi<0$, i.e. (b) is proved similarly. Remark \ref{rkpropabsury} is proved. 
\end{proof}
	

\subsubsection{The BSDE$(\xi,f)$ with $f(t,\omega, y,z):= \beta_t(\omega)y + \gamma_t(\omega)z -  \frac{\delta}{y}|z|^2$ }	
	
Let us proceed with the following result and existence for the BSDE$(\xi,f)$ which is analogous to Proposition \ref{propabsury}. 
Its proof uses Itô's formula and the domination argument Lemma \ref{lem:dom} as	in the proof of Proposition \ref{propabsury}.
The difficulty comes with the negative scaling $-\delta$.  
Note that this case is related to the Kreps-Porteus model. 
	
\begin{proposition}\label{prop:-absury}
	Let $\delta\ge0$ be fixed and such that $\delta \neq 1/2$ and assume that $\xi, \beta$ and $\gamma$ satisfy the integrability condition \ref{a1second}.
	The following statements hold:
	\begin{itemize}
		\item[(a)] If $\xi>0$, then the BSDE$(\xi,f)$ has a solution  $(Y, Z)$ such that 
		$\sup_{t \in [0,T]}\E_{\Q}[|Y_t|^{-2\delta + 1}]<\infty$, $Y>0$, and $Z\in {\cal L}^2(\mathbb{R}^d)$ where $\Q$ is the probability measure with density 
		\begin{equation}\label{qgamma}
			d\Q=\mathcal{E}\Big(\int\gamma\,dW\Big)\,d\P.
		\end{equation}
		\item[(b)] If $\xi<0$, with $-2\delta + 1$ being an odd number, then the BSDE$(\xi, f)$ has a solution $(Y,Z)$ such that\\
		$\sup_{t \in [0,T]}\E_{\Q}[|Y_t|^{-2\delta + 1}]<\infty$, $Y<0$, and $Z \in \mathcal{L}^2(\R^d)$.
	\end{itemize} 
\end{proposition} 
	\begin{proof}[of Propositions \ref{aperitif}-(b) and  \ref{prop:-absury}]
	We will prove only \emph{(a)}, since the proof of \emph{(b)} is the same (we additionally assume $-2\delta+1$ odd to ensure that $\xi^{-2\delta +1}$ is well-defined for $\xi< 0$).
	Here again, we would like to apply Itô's formula to $\frac{1}{-2\delta + 1}Y^{-2\delta + 1}$ and to its inverse $\{(-2\delta + 1)\bar Y\}^{1/(-2\delta + 1)}$ to argue that the BSDE$(\xi, f)$ admits a solution if and only if the BSDE
	\begin{equation}
	\label{eq:linear.-delta}
		\bar Y_t = \frac{1}{-2\delta  + 1}\xi^{-2\delta + 1}+ \int_t^T (-2\delta + 1)\beta_s\bar Y_s + \gamma_s\bar Z_s\,ds - \int_t^T\bar Z_s\,dW_s
	\end{equation}
	also admits a solution denoted $(\bar Y, \bar Z)$.
	And by the integrability condition on $\xi$, it follows by \cite[Theorem 2.1]{beh2015} that
	$(\bar Y, \bar Z)\in {\cal S}^1(\mathbb{R})\times{\cal M}^2(\mathbb{R}^d)$, showing that $Y\in {\cal S}^{-2\delta +1}(\mathbb{R})$. 
	However, here there are more than one cases since the number $\{-2\delta + 1\}^{1/(-2\delta + 1)}$ is not always defined. 

	\vskip 0.15cm 	
	\textit{Case 1. $-2\delta + 1 > 0$.} In this case, $\{-2\delta + 1\}^{1/(-2\delta + 1)}$ is a number and the linear equation \eqref{eq:linear.-delta} is well-posed, thus the result follows by It\^o's formula as explained above.	
	
	\vskip 0.15cm 	
	\textit{Case 2.  $\delta \in \N^* $.} 
	This situation can be treated as in
	\textit{Case 1}, since the number $\{(-2\delta + 1)\}^{1/(-2\delta + 1)}$ is again well defined. 
	In particular,  we obtain the existence in part (b) of Proposition \ref{aperitif} by taking $\delta =1$.
	To finish the proof of Proposition \ref{aperitif}(b), note that the proof of
	uniqueness in the class $(D)$ goes as that of
	\cite[Proposition 2.2]{Bahlali_domi}. 
	Assume that $\xi$ is square integrable. 
	Given a solution $(Y,Z)$ of \eqref{eq:z2/y.intro},
	Itô's formula applied to $Y^2$ shows that 
	\begin{equation}\label{y2}
		Y_t^2 = \xi^2 -2\int_t^T |Z_s|^2\,ds - \int_t^T Y_sZ_sdW_s.
	\end{equation} 
	Using Burkholder-Davis-Gundy's inequality, one obtains that
	\begin{equation}\label{supy2}
		\frac12\E\big[\sup_{s\leq T}|Y_s|^2\big] + \frac32 \E\Big[\int_0^T|Z_s|^2\,ds\Big] \ \leq \
		\E\big[|\xi|^2\big]	.
	\end{equation} 
	To prove uniqueness, let $(Y,Z)$ and $(Y',Z')$ be two solutions in
	$\mathcal{S}^2(\mathbb{R})\times\mathcal{M}^2(\mathbb{R}^d)$.
	Itô's formula applied to $|Y-Y'|^2$ shows that  
	\begin{equation}\label{y-y'2}
		|Y_t-Y'_t|^2 = |\xi-\xi'|^2 -2\int_t^T |Z_s-Z'_s|^2\,ds - \int_t^T
		(Y_s-Y'_s)(Z_s-Z_s')\,dW_s. 
	\end{equation}
	The uniqueness follows by using Burkholder-Davis-Gundy's inequality once again.

	
	\vskip 0.15cm 	
	\textit{Case 3. $1/2 < \delta < 1 $.} We will use Lemma \ref{lem:dom}. 
	Let $f_1(t,y,z):= \beta_ty + \gamma_tz- |z|^2/y$, $f_2(t,y,z):= \beta_ty + \gamma_tz$ and consider the three equations BSDE$(\xi,f_1)$, BSDE$(\xi,f)$ and BSDE$(\xi,f_2)$.
	Notice that $f_1\le f\le f_2$.
	The equations BSDE$(\xi,f_1)$ and BSDE$(\xi,f_2)$ admit solutions respectively denoted $(Y_1,Z_1)$ and $(Y_2, Z_2)$. 
	Observe that since $(Y_2,Z_2)$ solves a linear equation and that $Y_2$ is explicitly given by $Y_{2,t} = \E_\Q[e^{\int_t^T\beta_u\,du}\xi\mid \mathcal{F}_t]$. 
	Since $-Y_1^{-1}$ solves the linear equation \eqref{eq:linear.-delta} (with $\delta=1$), it follows that
	\begin{equation*}
	 	Y_{1,t} = \E_\Q\Big[e^{-\int_t^T\beta_u\,du}\xi^{-1}\mid \mathcal{F}_t\Big]^{-1}>0
	 \end{equation*} 
	 so that by Jensen's inequality we have $0<Y_1\le Y_2$ $\P$-a.s., by equivalence of the probability measures $\Q$ and $\P$. 
	Thus, by Lemma \ref{lem:dom}, BSDE$(\xi, f)$ has a solution $(Y,Z)$ such that $Y_1\le Y\le Y_2$.
	By the integrability condition \ref{a1second}, this shows that $\sup_t\E_{\Q}[|Y_t|^{-2\delta + 1}]<\infty$. 

	
	\vskip 0.15cm 	
	\textit{Case 4.  $\delta \geq 1 $.} 
	We denote by $\lfloor\delta\rfloor$
	the integer part of $\delta$. We have, for any $ \delta > 1$ and any $y>0$, 
	
	\begin{equation}\label{deltadomin}
		f_\delta:=\beta_ty+\gamma_t z-(\lfloor\delta\rfloor+1)z^2/y \leq \beta_ty+\gamma_t z-\delta z^2/y \leq \beta_ty+\gamma_t z . 
	\end{equation}
	By \emph{Case 2}, the BSDE$(\xi, f_\delta)$ admits a solution $(Y_\delta,Z_\delta)$ and the linear BSDE with generator $\beta_ty+\gamma_tz$ admits a solution $(Y_0, Z_0)$ with $Y_{0,t} = \E_\Q[e^{\int_t^T\beta_u\,du}\xi \mid \mathcal{F}_t]$.
	Since for any $\alpha > 0$, the function 
	$x \longmapsto x^{-\alpha}$ is convex on $\R_+^*$ and, as above, $Y_\delta$ satisfies
	\begin{equation*}
	 	Y_{\delta,t} = \E_\Q\Big[\exp\Big((-2\lfloor\delta\rfloor -1)\int_t^T\beta_u\,du\Big)\xi^{-2\lfloor\delta\rfloor -1} \mid \mathcal{F}_t \Big]^{\frac{1}{-2\lfloor\delta\rfloor -1}}>0,
	\end{equation*} 
	it follows  that $0 < Y_\delta\leq Y_0$ $\P$-a.s., by equivalence of the probability measures $\P$ and $\Q$.
	Using Lemma \ref{lem:dom}, we deduce that the BSDE driven by $(\xi,f)$ has a positive solution. 
\end{proof} 

	
\begin{remark}(Integrability)
	\label{moinsundemiline}
	Observe that we can derive stronger integrability properties for the solutions $(Y,Z)$ obtained in the previous proposition.
	In fact, we have the following two results:

	(i) \ Assume that conditions of Proposition \ref{prop:-absury} hold. 
	Then,  BSDE$(\xi,\ \beta_t y + \gamma_t z -  \frac12\frac{|z|^2}{y})$ has a solution $(Y, Z)$ which belongs to $\mathcal{S}^{2}(\R)\times\mathcal{L}^2(\R^d)$ and satisfies $Y > 0$. 
		
	(ii) \ Assume moreover that Assumption \ref{a2} is satisfied. 
	Then, $Z$ belongs to $\mathcal{M}^2(\R^d)$. 
\end{remark} 
\begin{proof} 
	According to Proposition \ref{prop:-absury},  BSDE$(\xi, \beta_t y + \gamma_t z - \frac{|z|^2}{y})$  has a solution $(Y_1,Z_1)$ such that $Y_1 > 0$. 	
	Let $f_1(t,y,z):= \beta_ty + \gamma_tz- |z|^2/y$, $f(t,y,z):= \beta_ty + \gamma_tz - \frac12 |z|^2/y$ and $f_2(t,y,z):= \beta_ty + \gamma_tz$.  Consider the three equations: BSDE$(\xi,f_1)$, BSDE$(\xi,f)$ and BSDE$(\xi,f_2)$.
	Clearly, $f_1(t,y,z)\le f(t,y,z)\le f_2(t,y,z)$ for any $y>0$ and any $t, z$.
	Arguing as in the proof of Proposition \ref{prop:-absury}-case 3, one can show that BSDE$(\xi, f)$ has a solution $(Y,Z)$ such that $0 < Y_1 \leq Y\leq Y_2$, where $(Y_1,Z_1)$ and $(Y_2, Z_2)$ are respectively the solutions of  BSDE$(\xi,f_1)$ and BSDE$(\xi,f_2)$.  
	Since $0 < Y_1 \leq Y \leq  Y_2$ and $Y_2$ belongs to $\mathcal{S}^2(\R^d)$, it follows that $Y$ belongs to $\mathcal{S}^2(\R^d)$. 
	We shall prove that $Z$ belongs to $\mathcal{M}^2(\R^d)$. Let $(Y, Z)$ be the solution we just  constructed. Itô's formula gives 
			\begin{equation*}
			Y^2_t = \xi^2 + \int_t^T \left(2\beta_s Y_s^{2} +  2Y_s\gamma_sZ_s \right) \,ds - \int_t^T |Z_s|^2\,ds   - 2\int_t^T Y_sZ_s\,dW_s.
		\end{equation*}
	Arguing as in the proof of Proposition \ref{propabsury}, we show that $Z$ belongs to $\mathcal{M}^2(\R^d)$. 
	Remark \ref{moinsundemiline} is proved.	 
\end{proof}		
	
	
\subsubsection{Proof of Theorems \ref{thm:exist_BSDE} and \ref{thm:existmoinsdelta} }
\begin{proof}[of Theorem \ref{thm:exist_BSDE}]
	The proof follows by the domination argument given in Lemma \ref{lem:dom}.
	Indeed, put $g(t, y, z) = \alpha_t + \beta_t y + \gamma_t z+  \frac{\delta}{y} |z|^2$.
	Since $\xi> 0$, by Proposition \ref{propabsury}, the BSDE$(\xi,g)$ admits a	solution $(Y^g, Z^g)$ such that $0<Y^g$.
	It follows by Lemma \ref{lem:dom} that BSDE$(\xi, H)$ admits a solution
	$(Y,Z)\in {\cal C}\times {\cal L}^2(\R^d)$ such that $0 <  Y\le Y^g$.
	In particular, this shows that if $\ref{a2}$ is satisfied, then $Y \in {\cal S}^{(2\delta  + 1)p}(\R^d)$ and if \ref{a3} is satisfied, then $\sup_{t\in [0,T]}\E[|Y_t|^{2\delta + 1}]<\infty$.
	Since $H\ge0$ and $\xi >0$, it holds $Y>0$.
\end{proof}
	
\begin{proof}[of Theorem \ref{thm:existmoinsdelta}]
	The proof is similar to the above argument. 
	Indeed, putting $f(t, y, z) =  \beta_t y + \gamma_t z-  \frac{\delta}{y}|z|^2$, when $\xi> 0$, we have $Y^f>0$ where $(Y^f, Z^f)$ solve the BSDE$(\xi,f)$, see Proposition \ref{prop:-absury} \emph{(a)}.
	Thus, it follows by Lemma \ref{lem:dom} that BSDE$(\xi, H)$ admits a solution $(Y,Z)\in {\cal C}\times {\cal L}^2(\R^d)$ such that $0 <  Y\le Y^f$.
	In particular, this shows that if $\ref{a1second}$ is satisfied, then $Y \in {\cal S}^{(2\delta  + 1)}(\R^d)$.
	The case $\xi<0$ is the same, in view of Proposition \ref{prop:-absury} \emph{(b)}.
\end{proof}
	
\subsection{Uniqueness}
The proof of uniqueness relies on convex dual representations, the proof of	which rely on stronger integrability properties for $Y$ and $Z$ which we now establish.
\begin{proposition}
\label{prop:Y bound Z BMO}
	Assume that $\xi\in \L^\infty$, $\alpha,\beta \in {\cal S}^\infty(\mathbb{R})$, and $\int\gamma\,dW\in  \mathrm{BMO}$.
	If $\xi > 0$ and $H$ satisfies $0\le H\le g$ then every solution $(Y,Z)$ of BSDE$(\xi, H)$ such that $0< Y \leq Y^{g}$  satisfies $Y \in {\cal S}^\infty(\mathbb{R})$ and $ Z \in \mathcal{L}^2(\R^d)$.

	Moreover, we have the following integrability properties for $Z$:
	\begin{itemize}
		\item[$(i)$] If $\delta < \frac12$ and $\int\gamma \,dW$ is a BMO martingale, then $\int Z\,dW$ is a BMO martingale.
		\item[$(ii)$] if $\big(\xi \geq c > 0$ \ and \ $\E\int_{0}^{T}(\alpha_{s} + \beta_{s} + |\gamma_s|^2)ds < \infty \big)$, then the stochastic integral $\int Z\, dW$ is a BMO martingale. 
	 \end{itemize}
\end{proposition}
\begin{proof}
	It follows by assumption that $\int\gamma\,dW$ satisfies the so-called reverse
	H\"older inequality, see e.g. \cite[Theorem 3.4]{Kazamaki01}.
	Thus, there is $q>1$ such that $\exp\big(\frac 12 \int_0^T\gamma_t\,dW_t \big)	\in \L^q$.
	Let  $(Y,Z)\in {\cal S}^2(\mathbb{R})\times {\cal L}^2(\mathbb{R}^d)$ be a solution of BSDE$(\xi, H)$ such that $0\leq Y \leq Y^{g}$, where $(Y^g,Z^g)$ is solution to BSDE$(\xi,g)$ with $g(t,y,z) = \alpha_t + \beta_ty + \gamma_tz + \frac{\delta}{y}|z|^2$.
	To prove boundedness of $Y$, it suffices to show that $Y^g$ is bounded.
	As shown in the proof of Proposition \ref{propabsury}, $0<Y^g \le \{(2\delta +	1)\bar Y\}^{1/(2\delta + 1)}$ with $\bar Y$ such that $(\bar Y, \bar Z)$ satisfies \eqref{eq:bsde_firstchange} for some progressive process $\bar Z$.
	Thus, since the generator of the BSDE \eqref{eq:bsde_firstchange} is bounded from above by $H_2(t, y, z) = (2\delta + 1)\big( \alpha_t + (\alpha_t +
		\beta_t)y \big) + \gamma_t z$, we have by Girsanov's theorem that
	\begin{equation*}
		\bar Y_t\le E_{Q^\gamma}\Big[\bar\xi + (2\delta + 1)\int_t^T\alpha_u +	(\alpha_u + \beta_u)\bar Y_u\,du\mid {\cal F}_t \Big].
	\end{equation*}
	This shows by Gronwall inequality with conditional expectation (see for instance \cite{Duf-Eps92}) that $\bar Y$ is bounded, and therefore that $Y^g$ is bounded.

{
	Let us now prove that $Z$ belongs to $\mathcal{M}^2$.

	$(i)$ We start by the case $\delta <1/2$. \ 
	Consider the function $f(x):=\frac{\delta} {x}$. 
	Set
	\begin{equation}
	\label{eq.def.K}
		K(y):= \int_0^y\exp\left(-2\int_1^xf(r)\,dr \right)\,dx.
	\end{equation}
	Observe that since $\delta<1/2$, the function $K$ is well--defined.
	Then the function
	\begin{equation}
	\label{eq.def.v}
		v(y):= \int_0^yK(x)\exp\left( 2\int_1^xf(r)\,dr\right)\,dx
	\end{equation}
	is increasing, belongs to ${\cal C}^1(\R)\cap{\cal C}^2(\R^*)$ and satisfies the ordinary differential equation $ \frac 12 v''(y)-f(y) v'(y)  = \frac 12$ in $\R^*$. 
	Let $\tau$ be a $[0,T]$-valued stopping time. 
	Since $Y>0$, then It\^o's formula applied to $v(Y)$ yields 
	\begin{align}
		\nonumber
		v(Y_\tau) &= v(\xi) + \int_\tau^{T}v'(Y_u)H(u, Y_u, Z_u) - \frac 12 v''(Y_u)|Z_u|^2\,du - \int_{\tau}^{T}v'(Y_u)Z_u\,dW_u\\
		  \nonumber
		       &\le v(\xi) + \int_\tau^{T}v'(Y_u)\left(\alpha_u + \beta_uY_u  + \frac 12\gamma^2_u \right)\,du\\
		&\quad - \int_\tau^T|Z_u|^2\left(\frac 12 v''(Y_u) - f(Y_u) v'(Y_u) \right)\,du - \int_\tau^{T}v'(Y_u)Z_u\,dW_u.
		\label{eq:ode estimate}
	\end{align} 
	From \eqref{eq:ode estimate}, we have
	\begin{align*}
	 	\frac 12\int_{\tau}^{T}|Z_u|^2\,du &\le v(Y_{T})-v(Y_{T})  + \int_\tau^{T}v'(Y_u)\left(\alpha_u + \beta_uY_u + \frac 12|\gamma_u|^2 \right)\,du - \int_\tau^{T}v'(Y_u)Z_u\,dW_u\\
	 	      &\le C_1 + C_2\int_\tau^{T} \frac 12|\gamma_u|^2 \,du + \int_\tau^{T}v'(Y_u)Z_u\,dW_u
	\end{align*}
	for some constants $C_1, C_2>0$ depending only on the uniform bounds of $\alpha$, $\beta$ and $\xi$.  
	\\ 
	Since $Y$ is continuous and bounded, $v'$ is continuous and $Z$ belongs to $\mathcal{L}^2$, 
	one can assume that the previous stochastic integral is a uniformly integrable martingale. Therefore, taking conditional expectation on both sides yields
	\begin{equation*}
		\E\Big[\int_\tau^{T}|Z_u|^2\,du\mid{\cal F}_{\tau} \Big] \le C_1 + C_2\E\Big[\int_\tau^{ T} |\gamma_u|^2\,du\mid{\cal F}_\tau\Big]
	\end{equation*}
	Therefore, $\int Z\,dW$ is a BMO martingale.

	$(ii)$ Let us now consider the case $\xi \geq c > 0$. 
	We first show that $Z \in \mathcal{M}^2(\R^d)$, and then use this to derive the BMO property.
	Let $C := \sup_{0\leq s\leq T}Y_s$. 
	Since  $\xi \geq c > 0$, it follows that the solution we constructed satisfies $Y_t  \geq c > 0$ for each $t$.
	Thus, using the occupation time formula, we have 
	\begin{align*}
		\E\Big[\int_0^T|Z_s|^2ds \Big] 
		& = \E\Big[\int_0^T d\left\langle Y\right\rangle_s\Big]  = \int_c^C \E[L_T^a(Y)]\,da
	\end{align*} 
	where $L^a_s(Y)$ the local time of the process $Y$ at $a$.
 	Arguing as in the proof of \cite[Propositions 2.1 and 2.2]{beo2017}, one can show that $\E(L_T^a(Y))$ is bounded uniformly in $a$, hence concluding the proof. 
 	For completeness, we provide below the detailled proof of this fact. 
 
	 Let  $\tau_{N}:=\inf \{t>0,\ \int_{0}^{t}|Z_{s}|^{2}ds\geq N\}$,  $\bar\tau_{M}:=\inf \{t>0,\
	 \int_{0}^{t}(\alpha_{s} + \beta_{s} + |\gamma_s|^2)ds\geq M\}$ 
	 and  put  $\tau := \tau_{N}\wedge \bar\tau_{M}$. Let $a$ be a nonnegative real number such that $c\leq a\leq C$ and $L_{\cdot}^{a}(Y)$ 
	 be the local time of $Y$ at the level $a$. By Tanaka's formula, we have
	 \begin{align*}
	 	(Y_{t\wedge \tau }-a)^{-}
	 	& =(Y_{0}-a)^{-} +\int_{0}^{t\wedge \tau }\mathbf{1}_{\{Y_{s}<a
	 		\}}dY_s+\frac{1}{2}L_{t\wedge \tau }^{a}(Y)
	 	\\
	 	& =(Y_{0}-a)^{-}-\int_{0}^{t\wedge \tau }\mathbf{1}_{\{Y_{s}<a\}}H(s,Y_{s},Z_{s})\,ds+\int_{0}^{t\wedge \tau }\mathbf{1}_{\{Y_{s}<a\}}Z_{s}dW_{s}+\frac{1}{2}L_{t\wedge \tau }^{a}(Y).
	 \end{align*}
	Since the map $y\mapsto (y-a)^{-}$ is Lipschitz, it follows that:
 	\begin{equation}  
 	\label{eq:LT}
 		\frac{1}{2}L_{t\wedge \tau }^{a}(Y)\leq |Y_{t\wedge \tau}-Y_{0}| + \int_{0}^{t\wedge \tau }\mathbf{1}_{\{Y_{s}\leq a\}}(\alpha_{s} + \beta_{s} + |\gamma_s|^2)ds  + \int_{0}^{t\wedge \tau }\mathbf{1}_{\{Y_{s}\leq a\}}\frac1c |Z_s|^2ds
 		 -\int_{0}^{t\wedge \tau }\mathbf{1}_{\{Y_{s}\leq a\}}Z_{s}dW_{s}.
 	\end{equation}
 	Passing to expectation, we get
 	\begin{equation}
 		\sup_{a}\mathbb{E}\left[ L_{t\wedge \tau }^{a}(Y)\right] \leq 2C+M+N(1 + \frac1c)
 	\label{estimationtempslocal}
 	\end{equation}
 	Since $c\leq Y_{t\wedge \tau }\leq C$ for each $t$, then $\mathrm{support}(L^a_{\cdot}(Y_{\cdot \wedge \tau }))\subset [c, \  C]$. 
 	Therefore using inequality \eqref{eq:LT} and the occupation time formula, we obtain
	\begin{align*}
 		\frac{1}{2}L_{t\wedge \tau }^{a}(Y) &\leq 2C + \int_{0}^{T} (\alpha_{s} + \beta_{s} + |\gamma_s|^2)ds+ \int_{0}^{t\wedge \tau }\mathbf{1}_{\{Y_{s}\leq a\}}(1 + \frac1c)d\langle Y\rangle _{s}
 		-\int_{0}^{t\wedge \tau}\mathbf{1}_{\{Y_{s}\leq a\}}Z_{s}\,dW_{s}\\
 		& \leq 2C + \int_{0}^{T} (\alpha_{s} + \beta_{s} + |\gamma_s|^2)ds\!+\!\!\int_{c}^{a}(1 + \frac1c)L_{t\wedge\tau}^{x}(Y)\,dx\!-\!\!\int_{0}^{t\wedge \tau }\!\!\mathbf{1}_{\{Y_{s}\leq a
 		\}}Z_{s}\,dW_{s}.
 	\end{align*}
 	Taking expectation, it holds that
	\begin{align*}
 		\frac{1}{2}\mathbb{E}\left[ L_{t\wedge \tau }^{a}(Y)\right]
 		& \leq 2C + \E\int_{0}^{T} (\alpha_{s} + \beta_{s} + |\gamma_s|^2)\,ds+\int_{c}^{a}(1 + \frac1c)\mathbb{E}\left[ L_{t\wedge\tau }^{x}(Y)\right]\, dx<\infty .
 	\end{align*}
 	Thanks to inequality (\ref{estimationtempslocal}) and Gronwall's lemma we deduce that
 	\begin{align*}
 		\mathbb{E}\left[ L_{t\wedge \tau }^{a}(Y)\right]
 		& \leq 2\bigg(2C + \E\Big[\int_{0}^{T}(\alpha_{s} + \beta_{s} + |\gamma_s|^2)ds\Big]\bigg)\exp \left( 2\int_{c}^{a}(1 + \frac1c)dx\right)\\
 		& \leq \bigg(4C + 2\E\int_{0}^{T}(\alpha_{s} + \beta_{s} + |\gamma_s|^2)ds\bigg)\exp (2(1 + \frac1c)(C-c)).
 	\end{align*}
 	Passing successively to the limit on $N$ and $M$ (keeping in mind that $\tau :=\tau_{R}\wedge \tau_{N}\wedge \bar\tau_{M}$) then using Beppo--Levi's theorem, we get
 	\begin{equation*}
 		\mathbb{E}\big[ L_{t}^{a}(Y)\big] \leq\bigg(4C + 2\E\int_{0}^{T}(\alpha_{s} + \beta_{s} + |\gamma_s|^2)ds\bigg)\exp (2(1 + \frac1c)(C-c)).   \notag  \label{Gronwall4}
 	\end{equation*}
 	This concludes the argument. 
} 

 	We prove the BMO property of $\int_0^t Z_s \, dW_s$. 
	Let $(Y,Z)$ be the solution which we have just constructed and $c$ be the non-negative constant which satisfies $Y_t \geq c$ for each $t$ an $a.s$. 
	Let the functions $\bar K$ and $\bar v$ be defined as in \eqref{eq.def.K} and \eqref{eq.def.v} respectively.
The function $\bar v$ belongs to ${\cal C}^1(\R)\cap{\cal C}^2(\R^*)$ and satisfies the ordinary differential equation $ \frac 12 v''(y)-\frac12 \frac{\delta}{y} v'(y)  = \frac 12$ in $\R^*$. 
Let $\tau$ be a $[0,T]$-valued stopping time.	Observe that  $\bar K(y)$ is positive for each $y \geq c$. Hence,  the derivative $\bar v'(y)$ is positive for each $y \geq c$. Since $Y_t \geq c$ for each $t$, $a.s$, then Itô's formula applied to $\bar v(Y)$ shows that	
\begin{align}
		\nonumber
		\bar v(Y_\tau) &= 	\bar v(\xi) + \int_\tau^{T}	\bar v'(Y_u)H(u, Y_u, Z_u) - \frac 12 	\bar v''(Y_u)|Z_u|^2\,du - \int_{\tau}^{T}	\bar v'(Y_u)Z_u\,dW_u\\
		\nonumber
		&\le 	\bar v(\xi) + \int_\tau^{T}	\bar v'(Y_u)\left(\alpha_u + \beta_uY_u  + \frac 12\gamma^2_u \right)\,du\\
		&\quad - \int_\tau^T|Z_u|^2\left(\frac 12 	\bar v''(Y_u) - f(Y_u) 	\bar v'(Y_u) \right)\,du - \int_\tau^{T}	\bar v'(Y_u)Z_u\,dW_u.
\label{eq:vbar}
\end{align}	
	Since $Z$ belongs to $\mathcal{M}^2(\R^d)$, $Y$ is bounded and $\bar v'$ is continuous,  it follows that the stochastic
	integral in the right hand side term of previous inequality is a uniformly integrable $\mathcal{F}_t$--martingale.	Passing to conditional expectation then arguing as in the proof of assertion (i), one can show that $\int Z\,dW$ is a BMO martingale. Proposition \ref{prop:Y bound Z BMO} is proved. 	
\end{proof}		

The next result gives a convex dual representation of the value process $Y$ of the solution $(Y,Z)$ of BSDE$(\xi, H)$.
We adapt the idea of \cite{tarpodual} to our situation.
However, we are dealing here with a solution, which does not necessarily coincide with the minimal solution studied in \cite{tarpodual}.
We denote by $H^*$ the convex conjugate of the function $H$ given by
\begin{equation}\label{Hstar}
	H^*(t,\omega, b,a) := \sup_{y,z}\Big(by + az - H(t,\omega,y,z)\Big).
\end{equation}
	
The function $H^*$ is convex and lower semicontinuous, and since $H$ is continuous, it can be checked that the function $H^*$ is progressively measurable, see e.g. \cite{tarpodual}.
Moreover, given a $\mathbb{R}^d$-valued progressively measurable process $a$ such that $\int a\,dW$ is in BMO.

\begin{proposition}
\label{prop:duality}
	Assume that $\xi\in L^\infty$, $\alpha, \beta \in {\cal S}^\infty$, that $\int\gamma \,dW$ is a BMO martingale and that for each $(t,\omega)$, the function $H(t,\omega,\cdot,\cdot)$ is jointly convex.
	Further assume that $\xi>0$ and $0\le H\le g$ holds.
	Then, every solution $(Y,Z)$ of BSDE$(\xi, H)$ such that $0<Y< Y^g$ admits the convex dual representation
	\begin{equation}
	\label{eq:duality}
		Y_t = \esssup_{a,b}\E_{\Q^a}\bigg[e^{\int_t^Tb_s\,ds}\xi - \int_t^Te^{\int_t^ub_s\,ds}H^*(u,b_u,a_u)\,du\mid {\cal F}_t \bigg],
	\end{equation}
	where the supremum is over progressively measurable processes $a:[0,T]\times \Omega \to \mathbb{R}^d$ and $b:[0,T]\times \Omega \to \mathbb{R}$ such that $\int a\,dW$ is in BMO and $|b|\le \sup_{0\le t\le T}|\beta_t|$.
		
	We have the same representation when $\xi<0$ and \ref{a2prime} holds, for all solutions between $Y^{-g}$ and $0$.
\end{proposition} 
\begin{remark}
	The integrability properties given in Proposition \ref{prop:Y bound Z BMO} remain true when $\xi<0$ and with minor changes due to replacing $\delta$ by $-\delta$, these properties hold when we have $\{\xi>0; -g\le H\le 0\}$ and $\{\xi<0; 0\le H\le -g\}$.
	Moreover, when $H$ is bounded from above (i.e. $g\le H\le 0$ or $-g\le H\le 0$), the representation given by Proposition \ref{prop:duality} remains valid if we replace the convexity assumption by concavity. 
	In this case, $esssup$ in the dual representation formula \eqref{eq:duality} is replaced by $essinf$.
\end{remark}

\begin{proof} [of Proposition \ref{prop:duality}]
	Let $(Y,Z)$ be a solution of BSDE$(\xi, H)$ such that $0< Y\le Y^{g}$.
	Let $b$ be a real-valued progressively measurable process satisfying $|b|\le \sup_{0\le t\le T}|\beta_t|$, and $a$ an $\mathbb{R}^d$-valued progressive process such that $\int a\,dW$ is in BMO.
	Applying It\^o's formula to $e^{\int_t^ub_s\,ds}\bar Y_u $, it follows by Girsanov change of measure that
	\begin{equation}
	\label{eq:first_inequ}
		Y_t = \E_{\Q^a}\bigg[ e^{\int_0^Tb_u\,du}\xi - \int_t^Te^{\int_t^ub_s\,ds}\big(b_u Y_u + a_u Z_u - H(u, Y_u, Z_u) \big)\,du\mid {\cal F}_t \bigg].
	\end{equation}
	It now follows by definition of $H^*$ and the fact that $a$ and $b$ were taken arbitrary that
	\begin{equation}
	\label{eq:first_inequality}
		Y_t \ge \esssup_{a,b}\E_{\Q^a}\bigg[e^{\int_t^Tb_u\,du}\xi - \int_t^Te^{\int_t^ub_s\,ds}H^*(u,a_u,b_u)du \mid{\cal F}_t \bigg].
	\end{equation}
	It remains to show that the above inequality is in fact an equality.
	Since $H(t,\omega,\cdot,\cdot)$ is convex, it has a nonempty subgradient at every point of the interior of its domain, see for instance \cite[Theorem 47.A]{Zeidler}.
	In particular, there is $\bar a(t,\omega)$, $\bar \beta(t,\omega)$ such that
	\begin{equation}
	\label{eq:subgrad}
		H(t,\omega, Y_t(\omega), Z_t(\omega)) = \bar b_t(\omega) Y_t(\omega) + \bar a_t(\omega)Z_t(\omega) - H^*(t,\omega, \bar b_t(\omega), \bar a_t(\omega)).
	\end{equation}
	It follows by measurable selection arguments that $\bar a$ and $\bar b$ can be chosen progressively measurable, see e.g. \cite{tarpodual} for details.
	Let us show in addition that $|\bar b|\le ||\sup_{0\le t\le T}|\beta_t|||_{\L^\infty}$.
	If $|\bar b|> ||\sup_{0\le t\le T}|\beta_t|||_{\L^\infty}$, then for all $y,z$, it holds
	\begin{equation}
	\label{eq:upHstar}
		H^*(t,\bar b_t,\bar a_t) \ge \bar a_t z + \bar b_t y - H(t,y,z)\ge \bar a_t z + \bar b_t y - \alpha_t - \beta_ty - \frac{\delta}{y}|z|^2.
	\end{equation}
	In particular, taking $z = 0$ and $y = nb_t$ yields
	\begin{equation*}
		H^*(t,\bar b_t, \bar a_t) \ge n|\bar b_t|(|\bar b_t| - \beta_t)-\alpha_t
	\end{equation*}
	which goes to infinity as $n$ goes to infinity.
	This is a contradiction, since $H$ being finite, $H^*(t,\bar b_t, \bar a_t)$ is also finite.
	Thus, $|\bar b_t|\le \sup_{0\le t\le T}|\beta_t|\in \L^\infty$.
	
	Now, let us denote by $\tau_n$ the exit time of $\int_t^\cdot|Z_u|^2\,du$ from the interval $[0,n]$, i.e.
		\begin{equation*}
			\tau_n : =\inf\Big\{s\ge t: \int_t^s|Z_u|^2\,du \ge n \Big\}\wedge T.
		\end{equation*}
		Since $Z$ is square integrable, it follows that $\tau_n\uparrow T$.
		By boundedness of $Y$, $\bar b$ and $\beta$, it follows that $\int_t^{T\wedge\tau_n}|\bar a_s|^2\, ds$ is bounded.
		Thus, putting $a^n:= a1_{[0,\tau_n]}$, it follows that $\int a^n\,dW$ is a BMO martingale and in particular $Q^{a_n}$ defines an equivalent probability measure.
		Applying It\^o's formula to $e^{\int_t^u\bar b_s\,ds}\bar Y_u $, it follows by Girsanov's change of measure that
		\begin{align*}
			Y_t &= \E_{\Q^{a^n}}\bigg[ e^{\int_t^{T\wedge \tau_n}\bar b_u\,du}Y_{T\wedge\tau_n} - \int_t^{T\wedge \tau_n}e^{\int_t^u\bar b_s\,ds}(\bar b_u Y_u + \bar a_u Z_u - H(u, Y_u, Z_u))\,du\mid {\cal F}_t \bigg]\\
			&= \E_{\Q^{a^n}}\bigg[ e^{\int_t^T b^n_u\,du} Y_{T\wedge\tau_n}  -
			\int_t^Te^{\int_t^ub^n_s\,ds} H^*(u, b^n_u,  a^n_u))\,du\mid {\cal F}_t
			\bigg].
		\end{align*}
		The second equality follows by \eqref{eq:subgrad}, with the notation $b^n:= b1_{[0,\tau_n]}$ and we used the fact that $H^*(u, b_u1_{[0,\tau_n]}, a_u1_{[0,\tau_n]}) = H^*(u, b_u, a_u)1_{[0,\tau_n]}$.
		Therefore, we have
		\begin{align*}
			Y_t 
			&\le \sup_{a, b}\E_{\Q^{a}}\bigg[ e^{\int_t^T b_u\,du}Y_{T\wedge \tau_n} -
			\int_t^Te^{\int_t^ub_s\,ds} H^*(u, b_u,  a_u)\,du\mid {\cal F}_t
			\bigg].
		\end{align*}
		Now, by continuity of $Y$, the sequence $Y_{T\wedge \tau_n}$ converges to $\xi$ $\P$-a.s. 
		We can now take the limit as $n$ goes to infinity on both sides using exactly the argument used to prove continuity in Lemma \ref{lem:continuous} below.
		Thus, obtaining 
		\begin{align*}
			Y_t &\le \sup_{a, b}\E_{\Q^{a}}\bigg[ e^{\int_t^T b_u\,du}\xi -
			\int_t^Te^{\int_t^ub_s\,ds} H^*(u, b_u,  a_u))\,du\mid {\cal F}_t
			\bigg].
		\end{align*}
		This concludes the proof.
\end{proof}
\begin{proof}[of Theorem \ref{thm:unique_BSDE}]
	The uniqueness of $Y$ follows from the representation \eqref{eq:duality} of
	Proposition \ref{prop:duality} and continuity of the paths of $Y$. 	
	We shall prove the uniqueness of $Z$.
	We call a pair $(Y', Z')$ a supersolution of BSDE$(\xi, H)$ if it satisfies $Y'_T\ge \xi$,
	\begin{equation*}
		Y'_s - \int_s^tH(u, Y'_u, Z'_u)\,du + \int_s^tZ'_u\,dW_u \ge Y'_t
	\end{equation*}
	for all $0\le s\le t\le T$ and the process $\int Z'\,dW$ is a supermartingale.
	Let $(Y,Z)\in {\cal S}^\infty(\mathbb{R})\times {\cal L}^2(\mathbb{R}^d)$ be a solution of BSDE$(\xi, H)$ such that $0< Y\le Y^g$ and $Z\in \mathcal{M}^2(\R^d)$ (by Proposition \ref{prop:Y bound Z BMO}).
	Let $(Y', Z') \in {\cal S}^\infty(\R)\times {\cal L}^2(\R^d)$ be an arbitrary supersolution.
	Let $a$ and $b$ such that $ a\in \mathcal{M}^2(\R^d)$ and $|b|\le \sup_{0\le t\le T}|\beta_t|$.
	Applying It\^o's formula to $e^{\int_t^ub_s\,ds}Y'_u$, it follows (using the arguments leading to \eqref{eq:first_inequality}) that
	\begin{equation*}
		Y'_t \ge \E_{\Q^a}\bigg[e^{\int_t^Tb_s\,ds}\xi -
			\int_t^Te^{\int_t^ub_s\,ds}H^*(s,b_s,a_s)\,ds\mid {\cal F}_t \bigg].
	\end{equation*}
	Since $a$ and $b$ were taking arbitrary, this implies
	\begin{equation*}
		Y'_t \ge \esssup_{a,b}\E_{\Q^a}\bigg[e^{\int_t^Tb_s\,ds}\xi - \int_t^Te^{\int_t^ub_s\,ds}H^*(s,b_s,a_s)\,ds\mid {\cal F}_t \bigg] = Y_t.
	\end{equation*}
	That is, $(Y, Z)$ is the minimal supersolution.
	According to \cite[Theorem 4.1]{DHK1101}, the solution $(Y,Z)$ is unique.
\end{proof}

\begin{proof}[of Corollary \ref{cor:cor.main.exit}]
	The existence part of the corollary follows by the domination argument
	Lemma \ref{lem:dom} exactly as in the proof of Theorem \ref{thm:exist_BSDE}, since we showed in Proposition \ref{prop:-absury} that BSDE$(\xi,-g)$ admits a solution.
	Uniqueness follows as in the proof of Theorem \ref{thm:unique_BSDE} since the convex dual representation remains true, with essential supremum replaced by essential infimum when $H$ is concave.
\end{proof}
	
\begin{remark}
	$(i)$ \	In the particular case of BSDE$(\xi,\frac \delta y|z|^2)$ comparison can be obtained without the above convex duality argument and boundedness of the terminal condition.
	In fact, if $\xi^1$ and $\xi^2$ are such that $\xi^1 \ge \xi^2>0$ and $(Y^i, Z^i)$ are the solutions of the BSDE$(\xi^i, \frac \delta y|z|^2)$, $i=1,2$.
	Putting $\bar Y^i = \frac{1}{2\delta + 1}(Y^i)^{2\delta + 1}$, it follows as in the proof of Proposition \ref{propabsury} that there is $\bar Z^i$ such that $(\bar Y^i, \bar Z^i)$ solves the BSDE
	\begin{equation*}
		\bar Y^i = \bar\xi^i  - \int_t^T\bar Z^i_s\,dW_s
	\end{equation*}
	with $\bar\xi^i:= \frac{1}{2 \delta + 1}(\xi^i)^{2\delta + 1}$.	
	Since $x\mapsto \frac{1}{2\delta + 1}x^{2\delta + 1}$ is increasing, it holds $\bar \xi^1 \ge \bar \xi^2>0$.
	Thus, it follows by classical comparison arguments that $\bar Y^1 \ge \bar Y^2$.
	Using that $x\mapsto \{(2\delta +1)x\}^{\frac{1}{2\delta + 1}}$ is increasing, we finally get $Y^1_t \ge Y^2_t$ for every $t \in [0,T]$.
		
	$(ii)$ Symmetrically, one can show an analogue result for BSDE$(\xi,-\frac \delta
		y|z|^2)$. 
\end{remark}
	
\section{Applications to parabolic PDEs}
\label{sec:PDE}

In this section we provide applications of our existence and uniqueness results to the study of singular parabolic PDEs.
We first consider semi-linear PDEs without boundary conditions, and then equations with lateral boundaries of Neumann type.
In this part, we assume that $\alpha$, $\beta$ and $\gamma$  are deterministic, i.e. depend only on $t$. 
We are concerned with the semilinear PDE associated to the Markovian version of	our BSDE.
	
\subsection{Probabilistic formulas for singular parabolic PDEs}
\label{sec:parab PDE}
	
Let us now turn to the application of our existence results to singular parabolic semilinear PDEs.
Let $\sigma $, ${\mu}$ be two measurable functions defined on $[0,T]\times\mathbb{R}^{d}$ with values in $\mathbb{R}^{d\times d}$ and $\mathbb{R}^{d}$ respectively.
Let $h$ be a measurable function defined on $\mathbb{R}^{d}$ with values in $\mathbb{R}$.
Define the differential operator ${\cal L}$ by
\begin{equation*}
	{\cal L}:=\sum_{i,\,j=1}^{d}(\sigma\sigma')_{ij}(s,x)\frac 12\frac{\partial^{2}}{\partial x_{i}\partial x_{j}}+\sum_{i=1}^{d}\mu_{i}(s,x)\frac{\partial }{\partial x_{i}}\, .
\end{equation*}
 We consider the followign equation:
\begin{equation} \label{edpinitiale}
	\begin{cases}
		\dfrac{\partial {v}}{\partial s}(s,\,x)+{\cal L}v(s,x)+H(s, v(s,\,x)), \sigma'\nabla_{x}v(s,\,x)) = 0,\ \text{on\ } [0,T)\times \mathbb{R}^{d},\,  \\
		v(T,x)=h (x).
		\end{cases}
\end{equation}
We will restrict ourselves to the case $\xi>0$ and $0\le H\le g$ leaving the other cases discussed in Corollary \ref{cor:cor.main.exit} to the reader, since they require only small modifications.
Consider the following conditions
\begin{enumerate}[label = (\textsc{A3}), leftmargin = 30pt]
	\item $\sigma:[0,T]\times\mathbb{R}^d\to \mathbb{R}^{d\times d}$ and ${\mu}:[0,T]\times\mathbb{R}^d\to \mathbb{R}^d$ are continuous functions and there exists $C>0$ such that for every $(t,x)\in [0,T]\times \mathbb{R}^d$ 
	\begin{equation*}
		|\sigma(t,x)| + |\mu(t,x)|\le C(1 + |x|). 
	\end{equation*}
	\label{a4}
\end{enumerate} 
		
\begin{enumerate}[label = (\textsc{A4}), leftmargin = 30pt]
	\item
	The SDE
	\begin{equation}
	\label{eq:sde}
		X_t = x + \int_0^t\mu(u,X_u)\,du + \int_0^t\sigma(u,X_u)\,dW_u
	\end{equation}
	admits a unique strong solution.
	\label{a5}
\end{enumerate}
\begin{enumerate}[label = (\textsc{A5}), leftmargin = 30pt]
	\item The terminal condition $h$ is continuous, bounded and satisfies $h>0$.
	\label{a6} 	
\end{enumerate}
\begin{theorem}
\label{viscosity}
	Assume that $\alpha, \beta$ and $\gamma$ are bounded measurable functions on $[0,T]$, that $H:[0,T]\times \mathbb{R}\times\mathbb{R}^d\to \mathbb{R}$ is jointly convex in the last two components $(y,z)$, $H$ satisfies $0\le H\le g$ and the conditions \ref{a4}-\ref{a6}
	are satisfied. Then, $v(t,x):=Y_{t}^{t,x}$ is a viscosity solution of the PDE (\ref{edpinitiale}).
\end{theorem} 
	
\begin{remark} \ Some comments on assumption \ref{a5} are in order.
	First recall that existence of a unique strong solution means that on the probability space $(\Omega, {\cal F}, \P)$, with the Brownian motion $W$, there is a $W$-adapted process $X$ satisfying \eqref{eq:sde}  and such that $X$ is indistinguishable to any other solution on this probability basis.
		
	General conditions on the coefficients $\mu$ and $\sigma$ guaranteeing existence of a unique strong solution of \eqref{eq:sde} are well known.
	Strong existence and uniqueness hold when $\mu$ and $\sigma$ are Lipschitz continuous, but also for much rougher coefficients.
	For instance, $\ref{a4}$ already implies \ref{a5} when $\sigma$ is a constant and non-zero, see \cite{MeMo17}.
	We further refer to \cite{bah2007, Ver81, KR05,Menou-Ouk-Tan} as well as \cite[Chapter 1]{Cher-Engl05} and the references therein for other conditions.
	The point here is that we obtain existence of \eqref{edpinitiale} under much weaker regularity conditions than the standard Lipschitz continuity conditions usually assumed.
\end{remark}
The following lemma will be needed to show that the function $v$ defined above is continuous.
\begin{lemma}
\label{lem:continuous}
	Assume that $H:[0,T]\times\mathbb{R}\times\mathbb{R}^d\to \mathbb{R}$ is convex in the last two components and that \ref{a2} is satisfied.
	Let $(\xi^{t,x})$ be a bounded family of $\sigma(W_r - W_t, t\le r\le T)$-measurable and strictly positive random variables.  Assume moreover that for every sequence $(t^n,x^n)$ converging to $(t,x)$, it holds $\xi^{t^n,x^n}\to \xi^{t,x}$ $\P$-a.s.
	Then, $(t,x)\mapsto Y_{t}^{t,x}$ is continuous, where $(Y^{t,x}_s, Z^{t,x}_s)_{s\in [t,T]}$ solves BSDE$( \xi^{t,x},H)$.
\end{lemma}
	
\begin{proof}
	Let $(t^n,x^n)\to (t,x)$, and assume without loss of generality that $t^n\downarrow t$.
	It follows by Proposition \ref{prop:duality} that for every $(t,x) \in [0,T]\times\mathbb{R}$, the solution $(Y^{t,x},Z^{t,x})$ of BSDE$(\xi^{t,x},H)$ admits the representation
	\begin{equation}
		\label{eq:duality-cont}
		Y_t^{t,x} = \esssup_{a,b}\E_{\Q^a}\bigg[e^{\int_t^Tb_s\,ds}\xi^{t,x} - \int_t^Te^{\int_t^ub_s\,ds}H^*(u,b_u,a_u)\,du\mid {\cal F}_t \bigg],
	\end{equation}
	where the supremum is over progressively measurable processes $a:[0,T]\times\Omega \to \mathbb{R}^d$ and $b:[0,T]\times \Omega \to \mathbb{R}$ such that
	$|b|\le \sup_{0\le t\le T}|\beta_t|$, and $\int a\,dW$ is in BMO.
	$Y^{t,x}_t$ is clearly deterministic, since $Y^{t,x}_s$ is measurable w.r.t. the sigma algebra $\sigma(W_r - W_t, t\le r\le s)$.
	Thus, \eqref{eq:duality-cont} takes the form
	\begin{equation}
	\label{eq:duality-cont-deter}
		Y_t^{t,x} = \sup_{a,b}\E_{\Q^a}\bigg[e^{\int_t^Tb_s\,ds}\xi^{t,x} -
		\int_t^Te^{\int_t^ub_s\,ds}H^*(u,b_u,a_u)\,du \bigg].
	\end{equation}
	By definition of $H^*$, for every $a,b$ it holds $H^*(u, b_u, a_u)\ge b_u - \beta_u - \alpha_u\ge -2||\beta||_{{\cal S}^\infty} -||\alpha||_{\infty} =:-C $ for some $C\ge0$.
	Thus by \eqref{eq:duality-cont-deter}, for every $n \in \mathbb{N}$, and every $a$, $b$, it holds
	\begin{equation*}
		Y^{t^n, x^n}_{t^n} \ge \E_{\Q^{a}}\bigg[e^{\int_{t^n}^Tb_s\,ds}\xi^{{t^n},{x^n}} - \int_{t^n}^Te^{\int_{t^n}^ub_s\,ds}(H^*(u,b_u,a_u)+C) - e^{\int_{t^n}^ub_s\,ds}C\,du \bigg].
	\end{equation*}
	Applying Beppo-L\'evy theorem and the Lebesgue   dominated convergence theorem, we get  $\liminf_{n\to \infty} Y^{t^n,x^n}_{t^n}\ge Y^{t,x}_t$.
		
	On the other hand, let $a^n, b^n$ be such that
	\begin{equation}
	\label{eq:1/n optimal Y}
		Y^{t^n,x^n}_{t^n} \le \E_{\Q^{a^n}}\bigg[e^{\int_{t^n}^Tb^n_s\,ds}\xi^{t^n,x^n} - \int_{t^n}^Te^{\int_{t^n}^ub^n_s\,ds}H^*(u,b^n_u,a^n_u)\,du \bigg] + \frac 1n.
	\end{equation}
	Since $\xi^{t^n,x^n}$ and $Y^{t^n,x^n}_t$ are bounded (see Proposition \ref{prop:Y bound Z BMO}), this implies
	$$\E_{\Q^{a^n}}\bigg[\int_{t^n}^Te^{\int_{t^n}^ub^n_s\,ds}H^*(u,b^n_u,a^n_u)\,du\bigg]\le C \quad\text{for some}\quad C\ge 0.$$
	Arguing as in the computations leading to \eqref{eq:upHstar}, and using that $\alpha, \beta$ and the sequence $b^n$ are bounded, we find two positive constants $C_1, C_2$ with $C_1>0$ such that
	\begin{equation*}
		H^*(t,b^n_t, a^n_t) \ge C_1|a^n_t|^2 + \delta b^n_t - C_2.
	\end{equation*}
	Hence, it follows that  $\E_{\Q^{a^n}}\big[\int_{t^n}^T\frac 12|a^n_u|^2\,du\big]\le C$ for some (possibly different) constant $C\ge 0$.
	This shows in particular, due to Girsanov's theorem, that $\E\big[\Gamma^{a^n}_{t^n,T}\log\big(\Gamma^{a^n}_{t^n,T} \big)\Big]\le C$, with $\Gamma^{a^n}_{t^n,T}: = \exp\big(\int_{t^n}^Ta^n_u\,dW_u - \frac12\int_{t^n}^T|a^n_u|^2\,du \big)$.
	Thus, by the criterion of de la Vall\'ee Poussin, $(\Gamma^{a^n}_{t^n,T})_n$ is uniformly integrable and therefore there exists $K \in L^1$ such that $(\Gamma^{a^n}_{t^n,T})_n$ converges weakly to $K$ (i.e. w.r.t. the topology $\sigma(L^1, L^\infty)$).
	Since the sequence $(\xi^{t^n,x^n} - \xi^{t,x})_n$ is uniformly bounded and converges to $0$ $\P$-a.s., letting $C\in \R$ be such that $e^{\int_{t^n}^Tb^n_s\,ds}\le C$, one has
	\begin{equation*}
		\limsup_{n\to \infty}\E\left[\Gamma^{a^n}_{t^n,T}e^{\int_{t^n}^Tb^n_s\,ds}|\xi^{t^n,x^n} - \xi^{t,x}| \right]\le C\lim_{n\to \infty}\E\left[\Gamma^{a^n}_{t^n,T}|\xi^{t^n,x^n} - \xi^{t,x}| \right]  = 0,
	\end{equation*}
	where the equality follows from \cite[Lemma 2.8]{Boue-Dup} (after pushing forward the measures $Q^{a^n}$ to the canonical space $C([0,T],\mathbb{R}^d)$).
	Hence, for every $\varepsilon>0$, there is $n$ large enough such that
	$$\E_{Q^{a^n}}\Big[e^{\int_{t^n}^Tb^n_s\,ds}\xi^{t^n,x^n}\Big]=
	\E\Big[Z^{a^n}_{t^n,T}e^{\int_{t^n}^Tb^n_s\,ds}\xi^{t^n,x^n} \Big] \le
	\E\Big[Z^{a^n}_{t^n,T}e^{\int_{t^n}^Tb^n_s\,ds}\xi^{t,x} \Big] + \varepsilon =
	\E_{Q^{a^n}}\left[e^{\int_{t^n}^Tb^n_s\,ds}\xi^{t,x} \right] + \varepsilon.$$
	That is, up to a subsequence, \eqref{eq:1/n optimal Y} yields
	\begin{align*}
		Y^{t^n, x^n}_{t^n} &\le \E_{\Q^{a^n}}\bigg[e^{\int_{t^n}^Tb^n_s\,ds}\xi^{t,x} - \int_{t^n}^Te^{\int_{t^n}^ub^n_s\,ds}H^*(u,b^n_u,a^n_u)\,du\bigg] + \frac 1n +
		\varepsilon\\
			&\le \sup_{a,b}\E_{\Q^a}\bigg[e^{\int_t^Tb_s\,ds}\xi^{t,x} - \int_t^Te^{\int_t^ub_s\,ds}H^*(u,b_u,a_u)\,du\mid {\cal F}_t \bigg] + \frac 1n + \varepsilon.
	\end{align*}
	Taking the limit as $n$ goes to infinity then letting $\varepsilon$ goes to $0$, we get $\limsup_{n\to\infty}Y^{t^n,x^n}_{t^n}\le Y^{t,x}_t$.
	This finishes the proof of continuity.
\end{proof}

\begin{proof}[of Theorem \ref{viscosity}]
	Let $H^*$ be the function defined by \eqref{Hstar}. For every $n\in \mathbb{N}$, the function
	\begin{equation}
	\label{eq:truncated H}
		H^n(t,y,z) := \sup_{\{|a|\le n; |b|\le n\}}\big(by + az - H^*(t, b, a)\big)
	\end{equation}
	is Lipschitz continuous (in $(y,z)$).
	Thus, by  \cite{peng01} the  BSDE driven by  $(h(X^{t,x}_T), H^n)$ admits a unique solution $(Y^{t,x,n}, Z^{t,x,n})$.
	Moreover, the function $v^n(t,x) := Y^{t,x,n}_t$ defines a viscosity solution of the PDE
	\begin{equation*}
		\begin{cases}
			\frac{\partial v^n}{\partial t}(t,x) + {\cal L}v^n(t,x) + H^n(t, v^n(t,x), \sigma'(t,x)\nabla_xv^n(t,x)) = 0, \quad \text{on } [0,T)\times \mathbb{R}^d,\\
			v^n(T,x) = h(x),
		\end{cases}
	\end{equation*}
	see e.g. \cite[Theorem 4.3]{peng01}.
	Note that the result of \cite{peng01} assumes Lipschitz continuity of $\mu$ and
	$\sigma$ in the $x$-variable to guarantee continuity of the function $v^n$.
	We emphasize that the result remains true when   $H^n$ is uniformly Lipschitz in $(y, z)$ and  $\mu$, $\sigma$, $h$ satisfies assumptions \ref{a4}-\ref{a6}. 
	Indeed, since for every $n$ it holds $v^n(t,x) := Y^{t,x,n}_t$ where $Y^{t,x,n}$ solves a BSDE with the convex generator $H^n$, continuity of $v^n$ follows by Proposition \ref{continuitytx} and Lemma \ref{lem:continuous} (applied to the generator $H^n$).
		
	We now prove that $v$ is a viscosity solution of Equation \eqref{edpinitiale}.
	Since $H$ is convex, proper and lower-semicontinuous, it follows from
	Fenchel-Moreau theorem (see e.g. \cite[Theorem 11.1, P. 474]{Rockafellar1998}) that $H^n\uparrow H$ pointwise.
	Thus, by continuity of $H$, we conclude from Dini's theorem that the sequence $(H^n)$ converges to $H$ uniformly on compacts subsets of $\mathbb{R}_+\times \mathbb{R}^d$.
	Next, we show that $(v^n)$ also converges to $v$ locally uniformly.
	Using the arguments of the proof of Proposition \ref{prop:duality} (and the fact that $Y^{t,x,n}$ is deterministic) allows to show that $Y^{t,x,n}$
	satisfies the dual representation
	\begin{equation*}
		Y^{t,x,n}_t = \sup_{b,a}\E_{\Q^a}\bigg[e^{\int_t^Tb_s\,ds}h(X^{t,x}_T) -\int_t^Te^{\int_t^ub_s\,ds}H^{n,*}(u, b_u, a_u)\,du \bigg]
	\end{equation*}
	where the supremum is taken over $\mathbb{R}^d$-valued progressively measurable processes $a$ such that $\int a\,dW$ is in BMO and $b$ a bounded adaped process, and $H^{n,*}$ is the convex conjugate of $H^n$ given by
	\begin{equation}\label{Hnstar}
		H^{n,*}(t,b,a) := \sup_{y\in \mathbb{R},z\in \mathbb{R}^d}\big(by + az -H^n(t, y,z)\big).
	\end{equation}
	Since $H^{n,*}\downarrow H^*$, it follows that $(Y^{t, x, n}_t)_n$ is increasing and $Y^{t,x,n}_t\le Y^{t,x}_t$, where $(Y^{t,x}, Z^{t,x})$ solves BSDE$(h(X^{t,x}), H)$.
	In particular, the sequence $(Y^{t, x, n}_t)_n$ converges.
	For any arbitrary admissible $(b,a)$, and every $n$, we have
	\begin{equation*}
		Y^{t,x,n}_t \ge \E_{\Q^a}\bigg[e^{\int_t^Tb_s\,ds}h(X^{t,x}_T) - \int_t^Te^{\int_t^ub_s\,ds}H^{n,*}(u, b_u, a_u)\,du \bigg]
	\end{equation*}
	so that by monotone convergence theorem and Proposition \ref{prop:duality} we get $\lim_{n\to\infty}Y^{t,x,n}_t\ge Y^{t,x}_t$.
	Thus,
	\begin{equation*}
		\lim_{n\to\infty}v^n(t,x) = \lim_{n\to\infty}Y^{t,x,n}_t = Y^{t,x}_t = v(t,x).
	\end{equation*}
	It follows again by Dini's theorem and continuity of $v$ (to get this, combine Lemma \ref{lem:continuous} and Proposition \ref{continuitytx}) that the convergence of $(v^n)$ to $v$ holds uniformly on compacts.
	By the stability theorem for viscosity solutions, see e.g. \cite[Lemma II.6.2]{Flem-Soner-second}, we conclude that $v$ is a viscosity solution of \eqref{edpinitiale}.
\end{proof}	
	
\subsection{The case of the canonical PDE $\theta\frac{|\nabla v|^2}{v}$ with $\theta=\pm1$}

As pointed out in Proposition \ref{aperitif}, in the special case where the generator is of the form $H(y,z):= \theta |z|^2/y$ (with $\theta=\pm1$), existence and uniqueness can be obtained under even weaker conditions than those given in Theorem \ref{viscosity}.
In the same vein, PDEs with nonlinearity of the form $\theta|\nabla v|^2/v$  can be treated with more general assumptions.
Since such equations are of particular interest in applications (see e.g. \cite{Dall-Lui-Pet}), we dedicate this subsection to their analysis.
Thus, for $\theta=\pm1$, consider the PDE 
\begin{equation} \label{edp1surY}
	\left\{
	\begin{array}{l}
		\dfrac{\partial {v}}{\partial s}(s,\,x)+\mathcal{L}v(s,x)+\theta\frac{|\sigma'\nabla_{x}v|^{2}}{v}(s,x) = 0, \qquad \ \  \text{on\ } [0,T)\times \mathbb{R}^{d},\,
			\\
			v(t,x) > 0
			\quad \text{on}\quad [0,T)\times \mathbb{R}^{d},\,
			\\
			v(T,x)= h(x).
		\end{array}
		\right.
\end{equation}
The following notions are well-known, they are recalled here for the reader's convenience.
	
\begin{definition}
	The SDE \eqref{eq:sde} admits a weak solution $(\bar X, \bar W)$ if there is a filtered probability space $(\bar \Omega,\bar{\cal F}, (\bar{\cal F}_t)_{t\in [0,T]}, \bar\P)$ on which $\bar W$ is an adapted Brownian motion and $\bar X$ is adapted and such that $(\bar X,\bar W)$ satisfies \eqref{eq:sde}.
		
	There is weak uniqueness (or uniqueness in law) if given any two weak solutions $(\bar X, \bar W)$ and $(\tilde X, \tilde W)$, possibly on different probability spaces, one has $\text{Law}(\bar X) = \text{Law}(\tilde X)$.
\end{definition}
Consider the following conditions:
\begin{enumerate}[label = (\textsc{A4'}), leftmargin = 30pt]
	\item	The SDE \eqref{eq:sde} admits a weak solution which is unique in law.
	\label{a5prime}
\end{enumerate}
\begin{enumerate}[label = (\textsc{A5}'), leftmargin = 30pt]
	\item The terminal condition $h$ is continuous, of polynomial growth and satisfies $h>0$. \label{a6prime}
\end{enumerate}
	
Let $(X_{s}^{t,x}, W_s)_{t\leq s \leq T}$ be the unique (weak) solution to SDE \eqref{eq:sde} starting from $(t,x)$ on the probability space $(\bar \Omega,\bar{\mathcal{ F}}_t,\bar \P)$.
Let $(Y_{s}^{t,x}, Z_{s}^{t,x})_{t\leq s \leq T}$ be the unique solution of BSDE$(h(X_{T}^{t,x}), |z|^2/y)$ on $(\bar \Omega, \bar{\mathcal{F}}_t,\bar \P)$.
\begin{theorem}
\label{thm:1/y viscosity}
	Assume that \ref{a4}, \ref{a5prime} and \ref{a6prime} hold.
	Then, $v(t,x):=Y_{t}^{t,x}$ is a viscosity solution of  PDE (\ref{edp1surY}).
\end{theorem}
\begin{remark}
\label{rem:no regula}
	$(i)$ \ Theorem \ref{thm:1/y viscosity} is specific for PDEs with the nonlinearity $|\nabla v|^2/v$.
	This restriction in the nonlinearity allows an important gain in the integrability of the terminal value as well as in the regularity to be imposed on the coefficients, since we only assume existence and uniqueness in law.
	For instance by a well-known result of \citet{Stroock-Varad} or \cite[Proposition 1.14]{Cher-Engl05}, Assumption \ref{a5prime} is already satisfied when $\mu$ and $\sigma$ are bounded continuous and such that  $\lambda^*\sigma(t,x)\sigma(t,x)^*\lambda|\ge \varepsilon|\lambda|^2$ for every $\lambda\in \mathbb{R}^d$ and some constant $\varepsilon > 0$.
		
	$(i)$ \ On the other hand, it should be noted that the literature dealing with singular PDEs of the form treated in this paper (which amount to initial value problems by a change of time) often assume the initial condition (corresponding to $h$ in our case) to be bounded and bounded away from zero.
	See e.g. \cite{Dall-Lui-Pet}, where further restrictions are made on the coefficient $\sigma$.	
\end{remark}
\begin{proof}[of Theorem \ref{thm:1/y viscosity}] 
	We treat the case $\theta=1$, the  case $\theta = -1$ goes similarly, with the function $u$ below replaced by the function $\widetilde u(x):= -\frac{1}{x}$. 
	Let $(X^{t,x}, \bar W)$ be a solution of \eqref{eq:sde} on a probability space $(\bar \Omega,\bar {\cal F}, \bar \P)$.
	Let $p>0$ be such that $|h(x)|\le C(1 + |x|^p)$ for some $C\ge0$.
	By Assumption \ref{a4} it is easily checked that $X_T^{t,x}\in L^{3p}(\bar\P)$, thus $h(X^{t,x}_T)\in L^3(\bar\P)$.
	Hence, by Proposition \ref{aperitif}, there is $(Y^{t,x}, Z^{t,x})$ solving BSDE$(h(X^{t,x}_T), |z|^2/y)$ driven by the Brownian motion $\bar W$ on  $(\bar \Omega,\bar {\cal F}, \bar \P)$.
	Notice that the function $v(t,x):=Y_t^{t,x}$ is well-defined and continuous.
	Indeed, put $u(y) := \frac13 y^3$, and let $(\widehat Y_s^{t,x}, \widehat Z_s^{t,x})$ be the unique solution of BSDE$(u(h(X_{T}^{t,x})), 0)$  in $\mathcal{S}^2(\mathbb{R})\times\mathcal{M}^2(\mathbb{R}^d)$. 
	It satisfies $\widehat Y_s^{t,x} = \E_{\bar\P}[u(h(X_{T}^{t,x})\mid \bar{\cal F}_s]= \E_{\bar\P}[u(h(X_{T}^{t,x})\mid X^{t,x}_s] $ where the second equality follows by the Markov property.
	Thus, $\widehat Y^{t,x}_t$ is deterministic.
	That is,
	$$
		v(t,x) = Y^{t,x}_t = u^{-1}(\widehat Y_t^{t,x}) = u^{-1}( \E_{\bar \P}[u(h(X^{t,x}_T))]) = u^{-1}\left(\int_{\mathbb{R}^d}u\circ h(r)\nu^{t,x}(dr)
		\right),
	$$
	where $\nu^{t,x}$ denotes the law of the solution $X^{t,x}_T$ which by \ref{a5prime} is unique.
	In particular, $v(t,x)$ does not depend on the underlying probability basis.
	By \ref{a5prime}, any other solution of \eqref{eq:sde} has the same law, showing that $v$ is well-defined.
	In particular, it does not depend on the probability basis on which $(Y^{t,x}_s, Z^{t,x}_s)$ is defined.
	Furthermore, since $u$ and $h$ are continuous, it follows by Proposition
	\ref{weakcontinuitytx} (in Appendix) that if $(t^n,x^n)\to (t,x)$, then $u\circ h(X^{t^n,x^n}_T) \to u\circ h(X^{t,x}_T)$ in law.
	Applying \cite[Theorem 6.1]{Pardoux99} shows that $\widehat{Y}^{t^n,x^n}$
	converges to $\widehat{Y}^{t,x}$ in law.
	Therefore,  $v(t^n,x^n) = u^{-1}\left(\widehat{Y}_t^{t^n,x^n} \right) \to u^{-1}\left(\widehat{Y}_t^{t,x} \right) = v(t,x)$. The rest of the proof follows exactly the proof of Theorem \ref{viscosity}, except that it is applied on the probability space $(\bar \Omega, \bar {{\cal F}}, \bar \P)$ with Brownian motion $\bar W$.
	Moreover, since the function $|z|^2/y$ is convex (but not continuous)  on  $\mathbb{R}\times \mathbb{R}^d$, we need to approximate it by a different sequence of Lipschitz continuous functions as follows:
	Consider the approximating sequence
	\begin{equation*}
		H^n(y,z) := \inf_{(y',z)\in (0,\infty)\times \mathbb{R}^d}(H(y',z') +
			n||(y,z)-(y',z')||).
	\end{equation*}
	Since $H$ is positive, it follows that for every $n$, $H^n>-\infty$ and so, the function $H^n$ is Lipschitz continuous, and we have $H^n\uparrow H$, see e.g. \cite[Example 9.11]{Rockafellar1998} for details.
	By continuity of $H$ on $(0,\infty)\times \mathbb{R}^d$, we conclude that $H^n$ converges to $H$ locally uniformly.
	\end{proof}
	
\subsection{Probabilistic formulas for singular parabolic PDEs with lateral Neumann boundary conditions}
\label{sec:neumann}

In this section, ${\cal O}\subseteq \mathbb{R}^d$ is an open, convex, connected and bounded subset of $\mathbb{R}^d$.
We assume that there is a function $\Phi\in {\cal C}^2_b(\mathbb{R}^d)$ such that ${\cal O} = \{\Phi >0\}$, $\partial {\cal O} = \{\Phi = 0\}$, $\mathbb{R}^d\setminus \overline{{\cal O}} = \{\Phi <0\}$ and $|\nabla_x \Phi(x)| = 1$ for all $x \in \partial {\cal O}$.
We consider the following parabolic PDE with lateral Neumann boundary conditions:
\begin{equation}
\label{eq:pde von neumann}
	\begin{cases}
		\frac{\partial v}{\partial t} + {\cal L} v + H(t,v, \sigma'\nabla v) = 0 \quad
		\text{for } (t,x) \in [0,T]\times {\cal O}\\
		\frac{\partial v}{\partial n} = 0 \quad \text{for } (t,x) \in [0,T]\times \partial {\cal O}\\
		v(T,x) = h(x) \quad \text{for } x \in \overline{{\cal O}},
	\end{cases}
\end{equation}
with
\begin{equation*}
	\frac{\partial}{\partial n}:= \sum_{i =1}^d\frac{\partial\Phi}{\partial x_i}\frac{\partial}{\partial x_i}.
\end{equation*}
	
Consider the reflected SDE
\begin{equation}
\label{eq:ref sde}
\begin{cases}
	X^{t,x}_s &= x + \int_t^s\mu(u,X^{t,x}_u)\,du + \int_t^s\sigma(u,X^{t,x}_u)\,dW_u + \int_t^s\nabla_x\Phi(X^{t,x}_u)\,dK^{t,x}_u\\
	K^{t,x}_s &= \int_t^s1_{\{X^{t,x}_u\in \partial {\cal O}\}}\,dK^{t,x}_u \quad
	\text{and } K^{t,x} \text{ is nondecreasing},
	\end{cases}
\end{equation}
with $(s,x)\in [t,T]\times {\cal O}$.
\begin{definition}\label{defrsde}
	A strong solution to the reflected SDE \eqref{eq:ref sde} is a pair $(X_{s}^{t,x}, K_{s}^{t,x})_{s\ge t}$ of $\mathcal{F}^W$--adapted processes satisfying equation \eqref{eq:ref sde}  and such that $$\int_t^T|\mu(s,X^{t,x}_u)|\,du + \int_t^T|\sigma(u, X^{t,x}_u)|^2\,du + \int_t^T|\nabla_x\Phi(X^{t,x}_u)|\,dK^{t,x}_u<\infty\quad \P\text{-a.s.}$$
	The solution $(X^{t,x}, K^{t,x})$ is said to be unique if it is indistinguishable to any other solution.
\end{definition}
Let us consider the assumption
\begin{enumerate}[label = (\textsc{A4''}), leftmargin = 30pt]
	\item The SDE \eqref{eq:ref sde} admits a unique strong solution.
	\label{a5second}
\end{enumerate}
We extend the solution process to $\left[ 0,T\right] $ by denoting
\begin{equation}
	X_{s}^{t,x}:=x,\;K_{s}^{t,x}:=0,\; \ \forall s\in\lbrack0,t).
	\label{extens X,K}
\end{equation}
It follows by the work of \citet{Lions-Szit84} that if $\mu:[0,T]\times\overline{{\cal O}}\to\mathbb{R}^d$ and $\sigma:[0,T]\times\overline{\cal O}$ are uniformly Lipschitz continuous, that is, there is $C\ge 0$ such that
\begin{equation*}
	|\mu(t,x) - \mu(t,y)| + |\sigma(t,x) - \sigma(t,y)| \le C|x - y| \text{ for all} x,y,
\end{equation*}
then there is a unique pair $(X^{t,x}_s, K^{t,x}_s)$ satisfying \eqref{eq:ref sde}.
Just as in the case of Equation \eqref{eq:sde}, assumption \ref{a5second} also holds under much weaker conditions, see e.g. \cite{T_Zhang94}.
	
Let $(Y^{t,x}, Z^{t,x})$ satisfy
\begin{equation*}
	Y^{t,x}_s = h(X_T^{t,x}) +\int_s^TH(u,Y^{t,x}_u, Z^{t,x}_u)\,du - \int_s^TZ^{t,x}_u\,dW_u,\quad s\in [t,T].
\end{equation*}
Then, we have:
\begin{theorem}
\label{thm:pde neumann}
	Assume that $\alpha, \beta$ and $\gamma$ are bounded measurable functions on
	$[0,T]$, that $H:[0,T]\times \mathbb{R}\times\mathbb{R}^d\to \mathbb{R}$ is jointly convex in $(y,z)$ and satisfies $0\le H\le g$, further assume that the conditions \ref{a4}, \ref{a5second} and \ref{a6} are satisfied.
	Then, $v(t,x):= Y^{t,x}_t$ is a viscosity solution of \eqref{eq:pde von neumann}.
\end{theorem}
\begin{proof}
	By Proposition \ref{continuitytxK} $X^{t^k,x^k}_T\to X^{t,x}_T$ in $\L^2$ whenever $(t^k,x^k)\to (t,x)$.
	Thus, due to Lemma \ref{lem:continuous}, the mapping $(t,x)\mapsto v(t,x):= Y^{t,x}_t$ is continuous.
		
	Let $H^k$ be the truncated generator defined in \eqref{eq:truncated H}.
	By \cite{Par-Zha98}, the function $v^k(t,x):= Y^{t,x,k}_t$ defines a viscosity solution of the PDE
	\begin{equation*}
		\begin{cases}
			\frac{\partial v^k}{\partial t}(t,x) + {\cal L}v^k(t,x) + H^k(t, v^k(t,x), \sigma'(t,x)\nabla_xv^k(t,x)) = 0, \quad \text{on } [0,T)\times {\cal O},\\
			\frac{\partial v^k}{\partial n} = 0 \quad \text{for } (t,x) \in [0,T]\times \partial {\cal O}\\
			v^k(T,x) = h(x),
			\end{cases}
		\end{equation*}
	where $(Y^{t,x,k}, Z^{t,x,k})$ is the unique solution of BSDE$(h(X^{t,x}_T),H^k)$.
	Note that the result of \cite{Par-Zha98} assumes Lipschitz continuity of $\mu$ and $\sigma$ to guarantee continuity of the function $v^k$.
	In the present article, this is obtained as in the proof of Theorem \ref{viscosity} when \ref{a4}, \ref{a5second} and \ref{a6} are satisfied.
	In fact, for every $k$, continuity of $v^k$ follows by Proposition \ref{continuitytx} and Lemma \ref{lem:continuous} (applied to the generator $H^k$).
		
	As argued in the proof of Theorem \ref{viscosity}, the sequence $(H^k)$ converges to $H$ locally uniformly, and $(v^k)$converges to $v$ locally uniformly.
	Let us use a stability argument to show that $v$ is a viscosity solution to equation \eqref{eq:pde von neumann}.
	We only prove the viscosity subsolution property.
	The viscosity supersolution property is proved similarly.
		
	Let $\phi \in \mathcal{C}^{1,2}_b$, (i.e. twice continuously differentiable with bounded derivatives) such that $v-\phi$ admits a local maximum at $(t_0,x_0)$.
	Without loss of generality, we can assume the maximum to be strict, see e.g.
	\cite[Chap. 2, Lemma 6.1]{Flem-Soner-second}.
	Let $B_\varepsilon(t_0,x_0)$ be a (closed) ball of radium $\varepsilon>0$ around $(t_0,x_0)$ where the maximum is realized.
	Denote by $(t_k,x_k)$ the point at which $v^k - \phi$ reaches its maximum in$B_\varepsilon(t_0,x_0)$.
	Let $(t,x) = \lim_k(t_k,x_k)$.
	Then, $(t_0,x_0) = (t,x)$.
	In fact, by uniform convergence of $v^k$ to $v$ on $B_\varepsilon(t_0,x_0)$, it holds $(v - \phi)(t_0,x_0) =\lim_k(v^k - \phi)(t_0,x_0)\le\lim_{k}(v^k-\phi)(t_k, x_k) = (v - \phi)(t,x)$, which implies that $(t_0,x_0) = (t,x)$ because otherwise $(t_0,x_0)$ is not the maximum.
	On the other hand, since $v^k$ is a viscosity solution to the previous PDE, it
	holds
	\begin{equation*}
		\begin{cases}
			\frac{\partial \phi}{\partial t}(t_k,x_k) + {\cal L}\phi(t_k,x_k) + H^k(t_k,\phi(t_k,x_k), \sigma'(t_k,x_k)\nabla_x\phi(t_k,x_k)) \ge 0, \quad \text{if }
			(t_k,x_k)\in [0,T)\times {\cal O},\\
			\max\left\{\left(\frac{\partial\phi}{\partial t} + {\cal L}\phi\right)(t_k,x_k) + H^k(t_k,\phi(t_k,x_k),\sigma'\nabla_x\phi(t_k,x_k)),\,\frac{\partial \phi}{\partial n}(t_k,x_k) \right\}\ge 0
				\quad \text{if } (t_k,x_k) \in [0,T]\times \partial {\cal O}.
		\end{cases}
	\end{equation*}
	If $x_0 \in {\cal O}$, then $\varepsilon>0$ can be chosen small enough so that $x_k \in {\cal O}$ for $k$ large enough.
	Thus,
	\begin{equation*}
		\frac{\partial \phi}{\partial t}(t_k,x_k) + {\cal L}\phi(t_k,x_k) + H^k(t_k,\phi(t_k,x_k), \sigma'(t_k,x_k)\nabla_x\phi(t_k,x_k)) \ge 0\quad \text{for all } n.
	\end{equation*}
	Since $(H^k)$ converges uniformly on $B_\varepsilon(t_0,x_0)$, it follows that
	\begin{equation*}
		\frac{\partial \phi}{\partial t}(t_0,x_0) + {\cal L}\phi(t_0,x_0) + H(t_0,v(t_0,x_0), \sigma'(t_0,x_0)\nabla_x\phi(t_0,x_0)) \ge 0.
	\end{equation*}

	If $x_0 \in \partial{\cal O}$, it cannot be guaranteed that $x_k\in \partial {\cal O}$ even for $k$ large enough.
	To overcome this difficulty, we will adapt a stability argument put forth in \cite{LionsDuke} for the first order Hamilton Jacobi equations to the present singular second order case.
	Let us assume by contradiction that there is $\delta>0$ such that
	\begin{equation}
	\label{eq:neumann contradiction}
		\max\bigg\{\bigg(\frac{\partial\phi}{\partial t} + {\cal L}\phi\bigg)(t_0,x_0) + H(t_0,\phi(t_0,x_0),\sigma'\nabla_x\phi(t_0,x_0)); \,\frac{\partial\phi}{\partial n}(t_0,x_0) \bigg\}=-\delta<0.
	\end{equation}	
	Let $\Phi$ be the function introduced at the beginning of the section.              
	For a fixed $\varepsilon'>0$, the function $v - \phi -\ varepsilon' \Phi$ has a strict local maximum at $(t_0,x_0)$, since for every $(t,x) \in  B_\varepsilon(t_0,x_0)$, it holds $\Phi(x_0)\le \Phi(x)$ and therefore $(v- \phi- \varepsilon'\Phi)(t,x) < (v-\phi - \varepsilon'\Phi)(t_0,x_0)$.
	Thus, as proved in the first part, $v^k - \phi - \varepsilon'\Phi$ admits a strict local maximum $(t_k,x_k)$ such that $(t_k,x_k)\to (t_0,x_0)$ and since	$v^k$ is a viscosity subsolution, we then have
	\begin{equation*}
		\begin{cases}
			\frac{\partial }{\partial t}(\phi+\varepsilon'\Phi)(t_k,x_k) + {\cal	L}(\phi+\varepsilon'\Phi)(t_k,x_k) + H^k(t_k, \phi+\varepsilon'\Phi(t_k,x_k), \sigma'(t_k,x_k)\nabla_x(\phi+\varepsilon'\Phi)(t_k,x_k)) \ge 0, \\
			\text{if } (t_k,x_k)\in [0,T)\times {\cal O},\\
			\max\left\{\left(\frac{\partial\phi+\varepsilon'\Phi}{\partial t} + {\cal
					L}(\phi+\varepsilon'\Phi)\right)(t_k,x_k) +
				H^k(t_k,\phi+\varepsilon'\Phi(t_k,x_k),\sigma'\nabla_x(\phi+\varepsilon'\Phi)(t_k,x_k)),
				\,\frac{\partial }{\partial n}(\phi+\varepsilon'\Phi)(t_k,x_k) \right\}\ge 0\\
				\text{if } (t_k,x_k) \in [0,T]\times \partial {\cal O}.
		\end{cases}
	\end{equation*}
	We claim that for $k$ large enough and $\varepsilon'$ small enough, it holds
	\begin{equation}
	\label{eq:lions idea}
		\frac{\partial }{\partial t}(\phi+\varepsilon'\Phi)(t_k,x_k) + {\cal L}(\phi+\varepsilon'\Phi)(t_k,x_k) + H^k(t_k, \phi+\varepsilon'\Phi(t_k,x_k),\sigma'(t_k,x_k)\nabla_x(\phi+\varepsilon'\Phi)(t_k,x_k)) \ge 0.
	\end{equation}
	In fact, if $x_k \in {\cal O}$ for every $k$, then \eqref{eq:lions idea} holds.
	If there is $k$ such that $x_k \in \partial {\cal O}$, assume that $\frac{\partial }{\partial n}(\phi+\varepsilon'\Phi)(t_k,x_k)\ge0$.
	Then since $\frac{\partial \Phi}{\partial n}=1$ on $\partial {\cal O}$, this implies $\frac{\partial \phi}{\partial n}\ge-\varepsilon'$. 
	But by \eqref{eq:neumann contradiction}, we have $\frac{\partial \phi}{\partial n}(t_k,x_k)\le -\delta$ (for $k$ large enough).
	Thus, choosing $\varepsilon'<\delta$ yields a contradiction, which proves the claim.
	Taking the limit in \eqref{eq:lions idea} first as $k$ goes to infinity and then the limit as $\varepsilon'$ goes to $0_+$, we have
	\begin{equation*}
		\bigg(\frac{\partial\phi}{\partial t} + {\cal L}\phi\bigg)(t_0,x_0) +	H(t_0,\phi(t_0,x_0),\sigma'\nabla_x\phi(t_0,x_0))\ge 0,
	\end{equation*}
	which contradicts \eqref{eq:neumann contradiction} and concludes the proof.
\end{proof}
\begin{remark}
	In view of Theorem \ref{thm:1/y viscosity} and its proof, the PDE 
	\begin{equation}
	\begin{cases}
		\frac{\partial v}{\partial t} + {\cal L} u +\frac{\sigma'|\nabla v|^2}{v} = 0
		\quad \text{for } (t,x) \in [0,T]\times {\cal O}\\
				\frac{\partial v}{\partial n} = 0 \quad \text{for } (t,x) \in [0,T]\times
				\partial {\cal O}\\
				v(T,x) = h(x) \quad \text{for } x \in \overline{{\cal O}},
		\end{cases}
	\end{equation}
	can also be solved when $h$ is continuous, has polynomial growth and $h>0$ and when $b$ and $\mu$ are two continuous functions of linear growth such that the SDE \eqref{eq:ref sde} admits a unique solution in law.
	The arguments are similar to those of Theorem \ref{thm:1/y viscosity} with minor modifications.
	\end{remark}
	
\section{Applications to decision theory}
\label{sec:decision}
In this final part we provide applications to decision theory in finance, including to expected utility maximization and to the existence of stochastic differential utilities.
	
\subsection{Utility maximization with multiplicative terminal endowment}
\label{sec:utility max}
	
We consider a market with $m$ stocks available for trading ($m\le d$) and following the dynamics
\begin{equation*}
	dS_t = S_t(b_t\,dt + \sigma_t\,dW_t), \quad S_0= s_0 \in \mathbb{R}^m_+
\end{equation*}
where $b$ and $\sigma$ are bounded predictable processes valued in
$\mathbb{R}^m$ and $\mathbb{R}^{m\times d}$, respectively.
We assume that the matrix $\sigma\sigma'$ is of full rank and define the $\mathbb{R}^d$-valued process $\theta:= \sigma'(\sigma\sigma')b$, the so-called market price of risk.
Let us denote by $\pi$ the trading strategy, i.e. $\pi^i_t$ is the part of total wealth invested in the stock $S^i$ at time $t$.
We denote by ${\cal A}$ the set of admissible trading strategies.
It is given by
\begin{equation*}
	{\cal A}:= \bigg\{\pi:[0,T]\times \Omega \to \mathbb{R}^m, \text{ progressive and } \int_0^T|\pi_t\sigma_t|^2\,dt<\infty\bigg\}.
\end{equation*}
Let $x>0$ be the initial wealth.
For every $\pi \in {\cal A}$, the wealth process $X^\pi$ given by
\begin{equation*}
	X_t^\pi = x + \int_0^t\sum_{i = 1}^d\frac{X^{\pi}_s\pi^i_s}{S^i_s}dS^i_s = x +	\int_0^tX^\pi_s\pi_s\sigma_s(\theta_s \,ds + dW_s)
\end{equation*}
is well-defined and positive.
To ease the notation, we put $p:= \pi\sigma$ and by abuse of notation we will write $p\in {\cal A}$.
It is well-known that in the above setting, the market is free of arbitrage, see e.g. \cite{DS94}.
In particular, $X^p$ is a local martingale under the equivalent probability	measure $\Q^\theta$.

The aim of this section is to solve the utility maximization problem from the terminal wealth of an investor with power or logarithmic utility functions and non-trivial terminal endowment $\xi $.
More precisely, we consider the utility maximization problem
\begin{equation}
\label{eq:problem_mult}
	V(x) = \sup_{p \in {\cal A}}E[U(X^p_T\xi)]
\end{equation}
where $U(x) = x^{1-\eta}/(1-\eta)$ with $\eta \in (0,1)$, or $U(x) = \log(x)$.
In \eqref{eq:problem_mult}, one can think of  $\xi$ as some random charge or tax that the investor is required to pay (or receive) for holding $X^p_T$.
That is, the investor pays/receives a terminal endowment $F$ in the form of a (random) portion of the terminal wealth, i.e. $F=X^p_T\Theta$.
In this case, the terminal utility becomes $U(X^p_T + F) = U(X^p_T(1 +\Theta))$, see e.g. \cite{Imk-Rev-Zha}.

\begin{proposition}[Power utility]
\label{prop:power utility}
	Assume that $U(x)= x^{1-\eta}/(1-\eta)$, $\eta \in (0,1)$.
	Further assume that $\eta$ is deterministic, $\xi \in L^\infty$  and $\xi>c$ for some $c>0$.
	Then, the value function is given by $V(x)=U(xY_0)$, where $Y_0$ is the initial value of a solution $(Y,Z)\in {\cal S}^2(\mathbb{R})\times {\cal M}^2(\mathbb{R}^d)$ of BSDE$(\xi, H)$, with
	\begin{equation}
	\label{eq:gen power}
		H(t,\omega,y,z) = 
		\begin{cases}
			\frac{1-2\eta }{2(1-\eta)}\frac{|z|^2}{y} -(1-\eta)\theta_tz+ \frac{1}{2\eta}\theta_t(1 - 2\theta_t)y \quad \text{if } y>0\\
				+\infty\quad \text{else}.
			\end{cases}
	\end{equation}
	Moreover, there exists an optimal admissible trading strategy $p^*$ given by
	\begin{equation*}
		p^*_t = \frac{1}{\eta }\Big(\theta_t + (1-\eta)\frac{ Z_t}{Y_t} \Big)\quad t\in [0,T]
	\end{equation*}
	and such that $\int p^*\,dW$ is a BMO martingale.
\end{proposition}
\begin{proof}
	The proof relies on application of the martingale optimality principle initiated by \cite{Hu-Imk-Mul} and the existence theorems for BSDEs derived above.
	Indeed, let us construct a family of processes $R^p$  such that for all $p \in {\cal A}$, $R^p_T = U(X^p_T  \xi)$; $R^p_0 = R_0$ does not depend on $p$; $R^p$ is a supermartingale for all $p \in {\cal A}$ and, there is $p^*\in {\cal A}$ such that $R^{p^*}$ is a martingale.
	It can be checked that if such a family is constructed, then $p^*$ is the optimal strategy and $R_0 = V(x)$ is the value function. See \cite{Hu-Imk-Mul} for details.
		
	Put $R^p_t:= U(X^p_tY_t)$ where $(Y,Z)$ is a solution to BSDE$(\xi, g)$ such that $c\le Y\le C$ for two strictly positive constants $c,C$ and $ Z\in \mathcal{M}$, for some function $g:[0,T]\times\Omega \times\mathbb{R}\times \mathbb{R}^d\to\mathbb{R}\cup\{+\infty\}$ to be determined.
	Applying It\^o's formula, we obtain
	\begin{align*}
		dR^p_t &= \left(U'(X^p_tY_t)\left\{X^p_tg(t,Y_t,Z_t) + Y_tp_tX^p_t\theta_t + X_t^pp_tZ_t\right\} + \frac 12 U''(X^p_tY_t)(X_t^p)^2|Z_t +p_tY_t|^2\right)\,dt\\
		&\quad + U'(X^p_tY_t)X^p_t(Z_t + Y_tp_t)\,dW_t.
	\end{align*}
	For all $n \in \mathbb{N}$, define the stopping time
	\begin{equation*}
		\tau_n:=\inf\bigg\{ t>0: \int_0^t\big( U'(X^p_sY_s)X^p_s(Z_s + Y_sp_s) \big)^2\,ds\ge n \bigg\}\wedge T.
	\end{equation*}
	Since $\int_0^{\cdot \wedge \tau_n}U'(X^p_sY_s)X^p_s(Z_s + Y_sp_s)\,dW_s$ is a martingale, it follows that if
	\begin{equation}
	\label{eq:condition g}
		g(t,Y_t, Z_t) \le -\left\{\frac{1}{2}\frac{U''(X^p_tY_t)}{U'(X^p_tY_t)}X^p_t|Z_t + p_tY_t|^2 + Y_t p_t\theta_t + p_tZ_t \right\}\quad \text{on } \{t\le \tau_n\},
	\end{equation}
	then the process $(R^p_{t\wedge \tau_n})_{t \in [0,T]}$ is a supermartingale.
	Thus, it follows from Fatou's lemma that $R^p$ is a supermartingale, since $\tau_n\uparrow T$ and $R^p$ has continuous paths.
	Since $U(x) = x^{1-\eta}/(1-\eta)$, the condition \eqref{eq:condition g} amongst to
	\begin{equation*}
		g(t, Y_t, Z_t)\le -\left\{\frac 12 \frac{\eta - 2}{Y_t}|Z_t + p_tY_t|^2 +Y_tp_t\theta_t + p_tZ_t \right\}.
	\end{equation*}
	We therefore put
	\begin{equation*}
		g(t,y,z):= \inf_{p\in \mathbb{R}^m}\Big( -\frac 12 \frac{\eta - 2}{y}|z + py|^2 - yp\theta_t - pz \Big).
	\end{equation*}
	A formal minimization shows that $g = H$ given by \eqref{eq:gen power} and  that the above infimum is attained by
	\begin{equation}
	\label{eq:p-start-mult}
		p^*_t = \frac{1}{\eta}\Big(\theta_t + (1-\eta)\frac{ Z_t}{Y_t}\Big).
	\end{equation}
	Since $Y\ge c$, the process $p^*$ is in $\mathcal{M}^2$.
	In particular, the candidate $p^*$ is admissible.
	Let us show that $R^{p^*}$ is a martingale.
	By construction, the drift term of $R^{p^*}$ is zero, so that
	$(R^{p^*}_{t\wedge \tau_n})_{t\in [0,T]}$ is a martingale.
	Thus, we can conclude by dominated convergence that $R^{p^*}$ is a martingale if we show that the set $\{R^{p^*}_\tau: \tau \text{ stopping time in } [0,T]\}$ is uniformly integrable.
	Since $\int p^*\,dW$ and $\int \theta\,dW$ are BMO martingales, it follows by \cite[Theorem 3.3]{Kazamaki01} that $\int p^*\,dW^{\delta\theta}$ is a BMO martingale under the probability measure $\Q^{(1-\eta)\theta}$, where $W^{(1-\eta)\theta}:=W - \int(1-\eta)\theta\,ds$ is a Brownian motion under $\Q^{(1-\eta)\theta}$.
	Thus, by \cite[Theorem 3.4]{Kazamaki01}, there is $q>1$ such that $X_T^{p^*}=\exp\Big( \int_0^Tp^*_t\,dW^{(1-\eta)\theta}_t-\frac12\int_0^T|p^*_t|^2\,dt\Big)\in \L^q(\Q^{(1-\eta)\theta})$.
	Since $\theta$ is bounded, we have ${\cal E}(-\int_0^T (1-\eta)\theta\,dW)\in\L^p$, with $1/p+1/q = 1$.
	Thus, there is a constant $C\ge0$ such that for every $[0,T]$-valued stopping time $\tau$, (putting $\delta:=1-\eta$) we have
	\begin{align*}
		\E\Big[ (R^{p^*}_\tau)^{1/(1-\eta)}\Big]&\le C\E\bigg[{\cal E}\bigg(\int_0^T\delta\theta\,dW\bigg){\cal E}\bigg(-\int_0^T\delta\theta\,dW\bigg)X^{p^*}_\tau \bigg]\le C\E_{\Q^{\delta\theta}}\bigg[ (X^{p^*}_\tau)^q\bigg]^{1/q}\E\bigg[{\cal E}\bigg(-\int_0^T\delta\theta\,dW\bigg)^{p} \bigg]^{1/p}\\
		&\le CE_{\Q^{\delta\theta}}\bigg[ (X^{p^*}_T)^q\bigg]^{1/q}\E\bigg[{\cal E}\bigg(-\int_0^T\delta\theta\,dW\bigg)^{p} \bigg]^{1/p}<\infty,
	\end{align*}
	where the second inequality follows by H\"older's inequality and the third one by Doob's maximal inequality.
	Therefore, $\{R^{p^*}_\tau: \tau \text{ stopping time in }[0,T]\}$ is uniformly integrable.
		
	It remains to show that BSDE$(\xi, g)$ admits a solution $(Y,Z)$ such that $0<c\le Y\le C$ for some positive constants $c,C$ and $\int Z\,dW$ is a BMO martingale.
	Put $\beta_t:=\frac{\theta_t(1 - 2\theta_t)}{2(1 - \delta)}$.
	Since $\theta$ is deterministic, the filtrations of  $W^{(1-\eta)\theta}$ and $W$ coincide. 
	Therefore,  BSDE$(\xi, g)$ admits a solution if and only if the equation
	\begin{equation}
	\label{eq:power bar bsde}
		\bar Y_t = e^{\int_0^T\beta_u\,du}\xi + \int_t^T\frac{1-2\eta}{2(1-\eta)}\frac{|\bar Z_u|^2}{\bar Y_u} \,du - \int_t^T\bar Z_u\,dW_u^{(1-\eta)\theta}\quad \Q^{(1-\eta)\theta}\text{-a.s.}
	\end{equation}
	admits a solution.
	If $\eta=1/2$, then \eqref{eq:power bar bsde} is clearly solvable.
	When $\eta<1/2$, the equation \eqref{eq:power bar bsde} admits a solution by Proposition \ref{propabsury} and when $\eta>1/2$, the existence follows by Proposition \ref{prop:-absury}.
	Since $\xi$ and $\eta$ are $\P$-a.s. bounded, it follows that the terminal condition of \eqref{eq:power bar bsde} is $\Q^{(1-\eta)\theta}$-a.s. bounded.
	Thus, the desired integrability properties of the solution $(\bar Y, \bar Z)$ follows by Proposition \ref{prop:Y bound Z BMO}
	since $\xi \ge c>0$.
\end{proof}
\begin{proposition}[log utility]
\label{prop:mult}
	Assume that $U(x)= \log(x)$.
	Further assume that $\xi \in \L^2$ and $\xi>0$. %
	Then, the value function is given by $V(x)=\log(xY_0)$, where $Y_0$ is the initial value of a solution $(Y,Z)$ of the BSDE$(\xi, H)$, where
	\begin{equation*}
		H(t,y,z) = 
		\begin{cases}
			\frac{1}{2}\frac{|z|^2}{y} - \frac{1}{2}\theta_t^2y\quad \text{if } y>0\\
				+\infty\quad \text{else},
		\end{cases}
	\end{equation*}
	with $\sup_{0\le t\le T}\E[|Y_t|^2]<\infty$ and $Z \in {\cal L}^2$.
	Moreover, there exists an optimal admissible trading strategy $p^*$ given by
	\begin{equation*}
		p^*_t = \theta_t\quad t\in [0,T].
	\end{equation*}
\end{proposition}
\begin{proof}
	The case of logarithmic utility follows exactly as in the proof of Proposition	\ref{prop:power utility}, except that here, $(Y,Z)$ is only assumed to be such that $Y>0$, $\sup_{t\le 0\le T}\E[Y^3_t]<\infty$ and $Z \in {\cal L}^2$. 
	When $U(x) = \log(x)$, the condition \eqref{eq:condition g} amongst to
	\begin{equation*}
		g(t,Y_t, Z_t) \le \frac 12\frac 1{Y_t} |Z_t + p_tY_t|^2 - Y_tp_t\theta_t - p_tZ_t.
	\end{equation*}
	Therefore, we put
	\begin{equation*}
		g(t, Y_t, Z_t) := \inf_{p \in {\cal A}}\left( \frac 12\frac 1{Y_t} |Z_t + p_tY_t|^2 - Y_tp_t\theta_t - p_tZ_t \right).
	\end{equation*}
	Since $Y>0$, the above infimum is attained by $p^*=\theta$ and $g(t, y, z)$ is given by
	\begin{equation*}
		g(t, y, z) = \frac{1}{2}\frac{|z|^2}{y} - \frac{1}{2}\theta_t^2y.
	\end{equation*}
	By Theorem \ref{thm:exist_BSDE}, the BSDE$(\xi, g)$ admits a solution $(Y,Z)$ such that $\sup_{t\le 0\le T}\E[| Y|^2_t]<\infty$ and $ Z \in {\cal L}^2$.
	By assumption, $p^*= \theta$ is admissible.
	It remains to show that $R^{p^*}$ is a martingale.
	This follows exactly as in the proof of Proposition \ref{prop:power utility}.
	In particular, we can use the same arguments to show that
	$\sup_\tau\E[\exp(R^{p^*}_\tau)]<\infty$, so that $\{R^{p^*}_\tau: \tau \text{stopping time in } [0,T]\}$ is uniformly integrable.
	This concludes the proof.
\end{proof}
	
\subsection{Stochastic differential utility}
\label{sec:SDU}

In economics and decision theory, Epstein-Zin preferences \cite{epstein02} refer to a class of (dynamic and) recursively defined utility functions (or preference specifications) given by
\begin{equation*}
	U_t(c) := F\Big(c_t, f^{-1}\Big(\E[f(U_{t+1}(c))\mid {\cal F}_t] \Big) \Big).
\end{equation*}
Here, $U_t(c)$ is the time-$t$ utility of the consumption $c=(c_t,c_{t+1},\dots)$, which is assumed to be an adapted sequence of real-valued random variables, $F:\mathbb{R}^2\to \mathbb{R}_+$ is a given function and $f:\mathbb{R}\to \mathbb{R}$ is a utility function, i.e. a strictly increasing and concave function. 
Epstein-Zin preferences are mostly important because they allow to disentangle risk aversion (modeled by $f$) and intertemporal substitution (modeled by $F$).
The continuous-time analogue (known as stochastic differential utility) of Epstein-Zin preferences was developed by \citet{epstein03} and defined as the unique adapted solution $(U_t)_{0\le t\le T}$ (when it exits) of the integral equation
\begin{equation}
\label{eq:sde xi = 0}
	U_t(c) = \E\Big[\int_t^Tg(c_s, U_s(c)) + \frac 12A(U_s(c))\sigma^2_U(s)\,ds\mid {\cal F}_t \Big].
\end{equation}
Here, $c:\Omega \times [0,T]\to \mathbb{R}$ is an adapted consumption process, $\sigma_U^2$ is the ``volatility'' of the unknown process $U$, and $A:\mathbb{R}\to \mathbb{R}$ and $g:\mathbb{R}\times \mathbb{R}\to \mathbb{R}$ are given functions modeling risk aversion and inter-temporal substitution, respectively.
When considering utility of a (terminal) position $\xi$ in addition to that of a consumption process $c$, the stochastic differential utility takes the form
\begin{equation}
\label{eq:sdu}
	U_t(\xi,c) = \E\Big[\xi+\int_t^Tg(c_s, U_s(\xi,c)) + \frac 12A(U_s(\xi,c))\sigma^2_U(s)\,ds\mid {\cal F}_t \Big].
\end{equation}
	
It was shown in \cite{epstein03} that when the function $A$ is continuous and integrable, the function $g(c,u)$ is Lipschitz continuous in $u$ and of linear growth in $c$, the integral equation \eqref{eq:sdu}, which of course coincides with the BSDE
\begin{equation*}
	dY_t = -\Big( g(c_t, Y_t) + \frac 12 A(Y_t)|Z_t|^2 \Big)\,dt - Z_t\,dW_t, \quad Y_T = \xi,
\end{equation*}
admits a unique square integrable solution.
We also refer to \cite{beo2017} for extensions of this result, for $A$ integrable.  
The case $A$ non globally integrable and $g$ continuous with linear growth has been studied in \cite{Bahlali_domi}.
Moreover, this class of utility functions are important in the context of asset pricing (see \citet{Duf-Eps92}), the case $A(u)\equiv -1/u$ being of particular interest, as the continuous time analogue of the Kreps-Porteus utility. 
This case is the subject of the work of Duffie-Lions \cite{Duf-Lions} who use a PDEs approach, in the Markovian setting. 
In fact, the Kreps-Porteus utility is obtained with the specifications
\begin{equation*}
	f_t(\omega,c,u) :=   \beta_t(\omega)u\quad \text{and}\quad
		A(u) := -\frac{\delta}{u} 
\end{equation*}
for some $\delta\ge0$ and $\beta$ two progressive processes.
The present work gives a BSDEs (and thus purely probabilistic) approach to this problem. 
	
A direct consequence of our main result is the following extension of the existence of a class of dynamic differential utilities of Kreps-Porteus type.
\begin{proposition}[Kreps-Porteus utility]
\label{prop:SDU}
	Let $\xi$ be strictly positive, assume that  $\xi,\beta$ satisfy the integrability condition 
	\begin{equation*}
	 	\exp\Big(\int_0^T|-\delta + 1||\beta_s| \,ds  \Big)\xi^{-2\delta + 1} \in \L^2.
	\end{equation*} 
	Put
	\begin{equation*}
		H_t(\omega,y,z):=-\frac{\delta}{2}\frac{|z|^2}{y} - \beta_t(\omega)y.
	\end{equation*}
	Then, there exists a stochastic differential utility $U \in {\cal S}^{(-\delta + 1)}(\mathbb{R})$ satisfying
	\begin{equation*}
		U_t(\xi,c) = \E\bigg[\xi+\int_t^Tf(c_s, U_s(\xi,c)) +\frac12A(U_s(\xi,c))\sigma^2_U(s)\,ds\mid {\cal F}_t \bigg].
	\end{equation*}
	Moreover, if $\xi$ and $\beta$ are bounded, then $U$ is bounded and unique among all bounded solutions with $\sigma_U\in \mathcal{M}^2$ .
\end{proposition}
\begin{proof}
	This result is a direct consequence of Corollary \ref{cor:cor.main.exit}.
\end{proof}
	
\begin{appendix}
\section{Continuity of SDE solutions w.r.t. initial parameters}
In this section we present the main arguments of the proofs of continuity results of SDE solutions in their initial conditions.
The first two results concern strong solutions of SDE with and without	reflection,  their proofs are modest extensions of the main result of \cite{bmo1998}.
The last result concerns a form of continuity for weak solutions.
This result seems to be new.		
\begin{proposition}\label{continuitytxK}
	Assume that \ref{a4} and \ref{a5prime} are satisfied, let $(t^n,x^n)$ be a	sequence in $[0,T]\times {\cal O}$ converging to $(t,x)$.
	If the pathwise uniqueness holds for Equation \ref{eq:ref sde}, then the sequence of processes $(X^{t_n,x_n}, K^{t_n,x_n})$ converges in $\mathcal{S}^2(\mathbb{R})$ to $(X^{t,x}, K^{t,x})$ which is the unique  solution of the SDE \eqref{eq:ref sde} starting at $x$ at time $t$.
\end{proposition}
\begin{proof}
	We follow the idea of the proofs given in  \cite{bmo1998, Kaneko-Nakao}, we also use some computations from \cite{Bah-Mal-Zal13}.
	Assume that the conclusion of Proposition \ref{continuitytxK} is false. 
	Then there exist a positive number $\delta$ and a sequence $(t_n, x_n)$ converging to $(t,x)$ such that
	\begin{equation}\label{contraryK}
		\inf_n\E\Big[\sup_{t\leq s\leq T}||(X_{s}^{t_n,x_n}, K_{s}^{t_n,x_n})	-(X_{s}^{t,x}, K_{s}^{t,x})||^2\Big]\geq \delta.
	\end{equation}	
	Without loss of generality, we assume that $t_n \geq t$ for each $n$.
	Arguing as in the proof of \cite[Lemma 3.8, p. 12]{Bah-Mal-Zal13}  and using assumption \ref{a4} and boundedness of the domain ${\cal O}$, we show that the sequence $(X^{t_n,x_n}, K^{t_n,x_n}, X^{t,x}, K^{t,x}, W)$ is tight in $C([0,T],\mathbb{R}^d)$. Hence, according to Skorohod's representation theorem, there exists a sequence of processes $(\bar X^n, \bar K^n,  \bar Y^n, \bar K^{1,n}, \bar W^n)_{n\geq 1}$ and a process $(\bar X, \bar K, \bar Y, \bar K^{1},\bar W)$ defined on some probability space $(\bar\Omega, \mathcal{\bar F}, \bar\P)$ such that for each $n$
	\begin{align}\label{barlawK}
		\text{Law} (\bar X^n, \bar K^n,  \bar Y^n, \bar K^{1,n}, \bar W^n) =
		\text{Law}(X^{t_n,x_n}, K^{t_n,x_n}, X^{t,x}, K^{t,x}, W)
	\end{align}	
	and there exists a subsequence still denoted $(\bar X^n, \bar K^{n}, \bar Y^n, \bar K^{1,n}, \bar W^n)$, such that
	\begin{align}\label{barconvergeK}
		\lim_{n\rightarrow \infty} (\bar X^n, \bar K^{n}, \bar Y^n, \bar K^{1,n}, \bar W^n) \ = \ (\bar X, \bar K, \bar Y, \bar K^{1},\bar W) \ \hbox{uniformly on
		every finite interval $\bar\P$-a.s.}
	\end{align}		
	Let $\bar {\cal F}_t^n$ (resp. $\bar {\cal F}_t$) be the $\sigma$-algebra $\sigma\big(\bar X_s^n, \bar Y_s^n, \bar W_s^n ; s\leq t\big)$ (resp. $\sigma\big(\bar X_s, \bar Y_s, \bar W_s ;  s\leq t\big)$) completed with $\bar P$-null sets. Hence $\big(\bar W_s^n, \bar {\cal F}_t^n\big)$ and $\big( \bar W_t,  \bar {\cal F}_t\big)$ are $\bar P$ Brownian motions and the processes $\bar X^n, \bar K^{n}, \bar Y^n, \bar K^{1,n}$ (resp.  $\bar X, \bar K, \bar Y, \bar K^{1}$) are adapted to $\bar {\cal F}_t^n$ and $\bar {\cal F}_t$ respectively.
			
	From \eqref{barlawK} and \eqref{eq:ref sde} we have
	\begin{equation}\label{SDEXnbarK}
		\begin{cases}
			\bar  X_{s}^{n} 
			&=x_n+\displaystyle\int_{t_n}^{s}\mu({r,\bar{X}_{r}^{n})}dr+\displaystyle\int_{t_n}^{s}\sigma(r, \bar X_{r}^{n})d \bar W_{r}^{n}+ \int_{t^n}^s\nabla_x\Phi(\bar X^{n}_r)\,d\bar K^{n}_r \quad \bar \P\text{-a.s.}\\
			\bar K^{n}_s &= \int_{t_n}^s1_{\{ \bar X^{n}_u\in \partial {\cal O}\}}\,d\bar K^{n}_u
		\end{cases}
	\end{equation}
	and
	\begin{equation}\label{SDEYnbarK}  
		\begin{cases}
			\bar  Y_{s}^{n} &=x+\displaystyle\int_{t^n}^{s}\mu({r,
			\bar Y_{r}^{n})}dr+\displaystyle\int_{t^n}^{s}\sigma(r, \bar Y_{r}^{n})d\bar
			W_{r}^{n} + \int_{t^n}^s\nabla_x\Phi(\bar Y^n_r)\,d\bar K^{1,n}_r \quad \bar \P
			\text{-a.s.}\\
			\bar K^{1,n}_s &= \int_{t_n}^s1_{\{ \bar Y^{n}_u\in \partial {\cal O}\}}\,d\bar K^{1,n}_u,
		\end{cases}
	\end{equation}
	see \cite{bmo1998} and \cite{Bah-Mal-Zal13} for details.
	Using \eqref{barconvergeK}, \eqref{SDEXnbarK} and \eqref{SDEYnbarK}, we show that
	\begin{equation}\label{SDEXbar}
		\begin{cases}
			\bar X_{s}
			&=x+\displaystyle\int_{t}^{s}\mu({r,\bar {X}_{r})}dr+\displaystyle\int_{t}^{s}\sigma(r, \bar X_{r})d \bar W_{r} + \int_t^s\nabla_x\Phi(\bar X^{n}_r)\,d\bar K_r \quad \bar \P \text{-a.s.}\\
			\bar K_s &= \int_{t}^s1_{\{ \bar X_u\in \partial {\cal O}\}}\,d\bar K_u
		\end{cases}
	\end{equation}
	and
	\begin{equation}\label{SDEYbar}  
		\begin{cases}
			\bar  Y_{s} 
			&=x+\displaystyle\int_{t}^{s}\mu({r,\bar
			Y_{r})}dr+\displaystyle\int_{t}^{s}\sigma(r, \bar Y_{r})d\bar W_{r} + \int_t^s\nabla_x\Phi(\bar Y_r)\,d\bar K^{1}_r \quad \bar \P \text{-a.s.}\\
			\bar K^{1}_s &= \int_{t}^s1_{\{ \bar Y_u\in \partial {\cal O}\}}\,d\bar K^{1}_u.
		\end{cases}
	\end{equation}
	Thus, $(\bar X, \bar K)$ and $(\bar Y, \bar K^{1})$ satisfy then the same SDE with the same Brownian motion and the same initial value. 
	Therefore the pathwise uniqueness property shows that $(\bar X, \bar K)$ and $(\bar Y, \bar K^{1})$ are indistinguishable.
	Returning back to \eqref{contraryK}, we use \eqref{barlawK}, \eqref{barconvergeK} and the uniform integrability to get
	\begin{align*}\label{contrary2}
		\delta & \leq \lim\inf_n\E\Big[\sup_{t\leq s\leq T}||(X_{s}^{t_n,x_n},
		K_{s}^{t_n,x_n}) -(X_{s}^{t,x}, K_{s}^{t,x})||^2\Big] \nonumber	\\
		& = \lim\inf_n \bar \E\Big[\sup_{t\leq s\leq T}||(\bar X_{s}^{n}, \bar K_{s}^{n}) - (\bar  Y_{s}^{n}, \bar K_{s}^{1,n}) ||^2\Big]
			\\
		& \leq  \bar \E\Big[\sup_{t\leq s\leq T}||(\bar X_{s}, \bar K_{s})- (\bar Y_{s}, \bar K_{s}^{1})||^2\Big] \nonumber
				\\
		& = 0 \nonumber
	\end{align*}
	which is a contradiction. The proof is finished.
\end{proof}
		
\begin{proposition}\label{continuitytx}
	Assume that \ref{a3} and \ref{a4} are satisfied, let $(t^n,x^n)$ be a sequence in $[0,T]\times \mathbb{R}^d$ converging to $(t,x)$.
	Then the sequence of processes $(X^{t_n,x_n})$ converges in $\mathcal{S}^2(\mathbb{R})$ to $(X^{t,x})$ which is the unique strong solution of the SDE \eqref{eq:sde} starting at $x$ at time $t$.
\end{proposition}	
The proof is similar (and simpler) than that of Proposition \ref{continuitytxK}.		
\begin{proposition}\label{weakcontinuitytx}
	Assume that  \ref{a4} is satisfied, let $(t^n,x^n)$ be a sequence in $[0,T]\times {\cal O}$ converging to $(t,x)$.
	If the uniqueness in law holds for Equation \eqref{eq:ref sde}, then the sequence of processes $(X^{t_n,x_n})$ converges in law   to $X^{t,x}$ which is the unique  solution (in law) of the SDE \eqref{eq:ref sde} starting at $x$ at time $t$.
\end{proposition}
		
\begin{proof}
	Without loss of generality, we assume that $t_n \geq t$ for each $n$.
	Using assumption \ref{a4} , we show that the sequence $(X^{t_n,x_n},  W)$ is tight in $C([0, T],\mathbb{R}^d)$. Hence, according to Skorohod's representation theorem, there exists a sequence of processes $(\bar X^n, \bar W^n)_{n\geq1}$ and a process $(\bar X, \bar W)$ defined on some probability space $(\bar\Omega, \mathcal{\bar F}, \bar\P)$ such that for each $n$
	\begin{align}\label{barlaw}
		\text{Law} (\bar X^n,   \bar W^n) = \text{Law}(X^{t_n,x_n}, W)
	\end{align}	
	and there exists a subsequence still denoted $(\bar X^n,   \bar W^n)$, such that
	\begin{align}\label{barconverge}
		\lim_{n\rightarrow \infty} (\bar X^n,  \bar W^n) \ = \ (\bar X, \bar W) \
		\hbox{uniformly on every finite interval $\bar\P$-a.s.}
	\end{align}		
	Let $\bar {\cal F}_t^n$ (resp. $\bar {\cal F}_t$) be the $\sigma$-algebra $\sigma\big(\bar X_s^n,  \bar W_s^n ; s\leq t\big)$ (resp. $\sigma\big(\bar X_s, \bar W_s ;  s\leq t\big)$) completed with $\bar P$-null sets. Hence $\big(\bar W_t^n, \bar {\cal F}_t^n\big)$ and $\big( \bar W_t,  \bar {\cal F}_t\big)$ are $\bar P$ Brownian motions and the processes $\bar X^n$ (resp.  $\bar X$) are adapted to $\bar {\cal F}_t^n$ and $\bar{\cal F}_t$ respectively.
			
	From \eqref{barlaw} and \eqref{eq:sde} we have
	\begin{equation}\label{SDEXnbar}
		\bar  X_{s}^{n}
		=x_n+\displaystyle\int_{t_n}^{s}\mu({r,\bar{X}_{r}^{n})}dr+\displaystyle\int_{t_n}^{s}\sigma(r, \bar X_{r}^{n})d \bar W_{r}^{n} \quad \bar \P \text{-a.s.}
	\end{equation}
	Using \eqref{barlaw}, \eqref{barconverge} and \eqref{SDEXnbar} we show that
	\begin{equation}\label{SDEXbar}
		\bar X_{s}=x+\displaystyle\int_{t}^{s}\mu({r,\bar {X}_{r})}dr+\displaystyle\int_{t}^{s}\sigma(r, \bar X_{r})d \bar W_{r}  \quad \bar \P \text{-a.s.}
	\end{equation}
	The  uniqueness in law shows that $\bar X$ and $X^{t,x}$ have the same law and that the whole sequence $X^{t_n,x_n}$ converges in law to  $X^{t,x}$.
\end{proof}
\end{appendix}


	\vspace{1cm}
	
	\noindent Khaled Bahlali: IMATH, Universit\'e de Toulon,  EA 2134,
	83957 La Garde Cedex, France.\\
	Acknowledgment: Partially supported by PHC Toubkal TBK/18/59. \\
	{\small \textit{E-mail address:} khaled.bahlali@univ-tln.fr}.\bigskip

	\noindent Ludovic Tangpi: Department of Operations Research and Financial
	Engineering, Princeton University, Princeton, 08540, NJ;  USA.\\
	Acknowledgment: Partially supported by NSF grant DMS-2005832.\\
	{\small\textit{E-mail address:} ludovic.tangpi@princeton.edu}.
	
\end{document}